\documentclass[11pt]{amsart}
\usepackage{amsmath, amsthm, amssymb}
\usepackage{graphicx}
\usepackage[usenames]{color}
\usepackage{srcltx} 
\usepackage{verbatim}

\newtheorem{thm}{Theorem}[section]
\newcommand{\bt}{\begin{thm}}
\newcommand{\et}{\end{thm}}

\newtheorem{ex}[thm]{Example}

\newtheorem{cor}[thm]{Corollary}   
\newcommand{\bc}{\begin{cor}}
\newcommand{\ec}{\end{cor}}

\newtheorem{lem}[thm]{Lemma}   
\newcommand{\bl}{\begin{lem}}
\newcommand{\el}{\end{lem}}

\newtheorem{prop}[thm]{Proposition}
\newcommand{\bp}{\begin{prop}}
\newcommand{\ep}{\end{prop}}

\newtheorem{defn}[thm]{Definition}
\newcommand{\bd}{\begin{defn}}    
\newcommand{\ed}{\end{defn}}

\newtheorem{rmrk}[thm]{Remark}   

\newcommand{\br}{\begin{rmrk}}
\newcommand{\er}{\end{rmrk}}

\newcommand{\GHto}{\stackrel { \textrm{GH}}{\longrightarrow} }
\newcommand{\Fto}{\stackrel {\mathcal{F}}{\longrightarrow} }

\newcommand{\be}{\begin{equation}}

\newcommand{\ee}{\end{equation}}

\newcommand{\N}{\mathbb{N}}

\newcommand{\diam}{\operatorname{Diam}}

\newcommand{\set}{\rm{set}}

\newcommand{\mass}{{\mathbf M}}

\begin{document}

\title[Contrasting Notions of Convergence]
{Contrasting Various Notions of Convergence in Geometric Analysis}

\author{Brian Allen}
\address{University of Hartford}
\email{brianallenmath@gmail.com}

\author{Christina Sormani}
\thanks{C. Sormani was partially supported by NSF DMS 1612049.}
\address{CUNY Graduate Center and Lehman College}
\email{sormanic@gmail.com}


\keywords{}



\begin{abstract}
We explore the distinctions between $L^p$ convergence of metric tensors 
on a fixed Riemannian manifold versus Gromov-Hausdorff, uniform, and intrinsic flat convergence
of the corresponding sequence of metric spaces.   We provide a number of
examples which demonstrate these notions of convergence do not agree even for
two dimensional warped product manifolds with warping functions converging in 
the $L^p$ sense.  We then prove a theorem which requires $L^p$ bounds from above
and $C^0$ bounds from below on the warping functions to obtain enough control for
all these limits to agree.
\end{abstract}

\maketitle

\section{Introduction}\label{Intro}

When mathematicians have studied sequences of Riemannian manifolds arising naturally in questions of almost rigidity or when searching for solutions to geometric partial differential equations, they have obtained bounds on the metric tensors of these Riemannian manifolds.  When the bounds they obtained on $(M^n,g_j)$ guaranteed a subsequence, $g_j \to g_\infty$ converging in the $C^0$ sense or stronger, then the
Riemannian manifolds, $(M,g_j)$, viewed as metric spaces, $(M,d_j)$, converge uniformly to $(M,d_\infty)$ where $d_\infty$ is defined as the infimum of the lengths of curves between points measured using $g_\infty$.  After observing this in \cite{Gromov-metric}, Gromov introduced the Gromov-Hausdorff distance between metric spaces, proving that uniform convergence implies Gromov-Hausdorff convergence of metric spaces.  The advantage of Gromov-Hausdorff convergence is that one may allow the spaces themselves to change $(M_j, d_j)$ and one may obtain a limit metric space which is not even a manifold.  Gromov proved that if $(M_j,g_j)$ have uniform lower bounds on Ricci curvature and uniform upper bounds on diameter than a subsequence converges in the Gromov-Hausdorff sense to a metric space in \cite{Gromov-metric} and since then many people have analyzed the properties of these limit spaces.

More recently the second author and Wenger introduced the intrinsic flat distance between oriented Riemannian manifolds which need not be diffeomorphic \cite{SW-JDG}.  Roughly the intrinsic flat distance is measuring a filling volume between two manifolds.  A standard sphere and a sphere with a thin deep well are very close in the intrinsic flat sense based on the filling volume of the well, while they are far apart in the Gromov-Hausdorff distance based on the depth of the well.   As soon as this notion was introduced people began asking whether $L^p$ convergence of the metric tensors might in some way be related to intrinsic
flat convergence of the metric spaces.   After all, a uniform $L^n$ bound on metric tensors implies a uniform upper bound on volume.   Wenger proved that as long as a sequence of oriented Riemannian manifolds has a uniform upper bound on volume and on diameter it has a subsequence converging in the intrinsic flat sense in \cite{Wenger-compactness}.  However it is not known whether the limit space is in anyway related to $(M, g_\infty)$ even when $g_\infty$ was smooth.   In joint work with Lakzian \cite{Lakzian-Sormani-1}, and work 
of Lakzian alone \cite{Lakzian-diameter} it was shown that even when $g_j \to g_\infty$ smoothly away from a singular set, the Gromov-Hausdorff and Intrinsic Flat limits need not be closely related to $(M, g_\infty)$ unless one controls volumes, areas, and distances near the singular set.

In this paper we provide a number of examples demonstrating that when metric tensors $g_j$ converge in the $L^p$ sense to a metric tensor $g_\infty$, then uniform, intrinsic flat and Gromov-Hausdorff limits need not converge to a metric space which is defined by $g_\infty$
using the infimum of lengths over all curves.  Our examples include very simple
two dimensional warped product Riemannian manifolds whose metric tensors are of the form $
dr^2 + f_j(r)^2 d\theta^2.   
$

In Example~\ref{Cinched-Torus} we find a sequence of warping functions $f_j(r)$
which converge in the $L^p$ sense to a constant function, $f_\infty$, but the uniform, Gromov-Hausdorff, and Intrinsic flat limit of the sequence is not even a Riemannian manifold.
In this example the $f_j\le f_\infty$ but have an increasingly narrow dip downward about $r=0$ so we say the sequence of manifolds is ``cinched'' at $0$.  This is an example with smooth convergence away from a singular set that was not seen in 
\cite{Lakzian-Sormani-1}.  The limit metric space is
described in detail within the example and a proof is given afterwards.    
In Example~\ref{MovingCinched} the $f_j \le f_\infty$
and $L^p$ converge to $f_\infty$ again, but the cinch moves around so that the $f_j$ do not converge pointwise almost everywhere.  This example has no uniform, Gromov-Hausdorff, or Intrinsic Flat limit unless one takes a subsequence where the cinch's location converges.

In Examples~\ref{SingleRidge}-~\ref{ManyRidges} we also consider warping functions, $f_j$, that $L^p$ converge to a constant function, $f_\infty$,  but now $f_j \ge f_\infty$.  In Example~\ref{SingleRidge} we have a single increasingly narrow peak about $r=0$.  We say there is a ``ridge'' at $0$.   This is another example with smooth convergence away from a singular set that was not studied in \cite{Lakzian-Sormani-1}.  We observe how the shortest paths between points on the ridge, do not lie on the ridge in Lemma~\ref{AvoidPeak2}.   In Example~\ref{MovingRidges} we have a sequence of manifolds with moving ridges, so there is no pointwise convergence almost everywhere.  In Example~\ref{ManyRidges} we have increasingly many increasingly dense ridges.  In all three of these examples we prove 
uniform convergence of the distances, $d_j$, to $d_\infty$ of the isometric product Riemannian manifold with
metric tensor $g_\infty= dr^2 + f_j(r)^2 d\theta^2$.  We obtain intrinsic flat and Gromov-Hausdorff convergence to this limit as well. 

In Example~\ref{to-R-ET} we have $f_j \ge f_\infty$ with $f_\infty$ constant and $f_j = f_\infty$
on an increasingly dense set.  However, now our $f_j$ do not converge in $L^p$ to $f_{\infty}$.  For the
particular sequence we chose, we obtain uniform, intrinsic flat and Gromov-Hausdorff convergence to a nonRiemannian Finsler manifold we call a minimized R-stretched Euclidean taxi metric
space.  This metric is defined as an infimum over an interpolation
between a Euclidean metric stretched by $R$ in one direction and a taxi metric.
Our example demonstrates that the $L^p$ convergence was crucial in the prior examples. As discussed in Remark \ref{RemScalarCompactness}, this example shows the necessity of scalar curvature bounds in the statement of the scalar compactness conjecture of Gromov-Sormani \cite{IAS} to conclude that the limit has Euclidean tangent cones almost everywhere. This conjecture was recently verified in the rotationally symmetric case by Park-Tian-Wang \cite{PTW}.

We then prove the following general theorem concerning warped product manifolds $M^n=[r_0,r_1]\times_f \Sigma$ where $\Sigma$ is an $n-1$ dimensional manifold including also $M$ without boundary that have $f$ periodic with period $r_1-r_0$ as in {\ref{MorN}):

\begin{thm}\label{WarpConv} 
  Assume the warping factors, $f_j\in C^0(r_0,r_1)$ , satisfy the following:
  \be
 0<f_{\infty}(r)-\frac{1}{j} \le f_j(r) \le K < \infty 
  \ee
and
  \be
  f_j(r) \rightarrow f_{\infty}(r)> 0 \textrm{ in }L^2
  \ee
  where $f_\infty \in C^0(r_0,r_1)$.
  
  Then we have GH and $\mathcal{F}$ convergence of the 
  warped product manifolds,
  \begin{align}
  M_j=[r_0,r_1]\times_{f_j} \Sigma &\to   M_\infty=[r_0,r_1]\times_{f_\infty} \Sigma,
  \\N_j=\mathbb{S}^1\times_{f_j} \Sigma &\to   N_\infty=\mathbb{S}^1\times_{f_\infty} \Sigma,
  \end{align}
  and uniform convergence of their distance functions, $d_ j \to d_\infty$.
 \end{thm}
 
\begin{rmrk} 
In our theorem we assume $L^2$ convergence but since we are assuming that the $f_j$ are uniformly bounded this is equivalent to $L^p$, $p \in [1,\infty)$ convergence.
\end{rmrk}
 
 The proof of this theorem and indeed the proof of all the examples relies on a theorem of the second author with Huang and Lee in the appendix of \cite{HLS} which is reviewed in the background section of this paper.  The theorem in \cite{HLS} states that if one has uniform upper and lower bounds on the $d_j$, a subsequence of the Riemannian manifolds converges in the uniform, Gromov-Hausdorff, and intrinsic flat convergence sense to some common limit space.  Thus we need only prove pointwise convergence of the original sequence of $d_j$ to our proposed $d_\infty$.   The method applied to control $d_j$ is different in each proof in this paper.  For the theorem, we apply the $C^0$ lower bound to bound $d_j$ from below and the $L^p$ upper bound is all that is needed to bound $d_j$ from above pointwise.  Note that the hypothesis of the theorem immediately implies a
 uniform upper bound on diameter [Lemma~\ref{lem-diam}].
We end the paper with Theorem~\ref{WarpConv2} concerning warped product manifolds where the warping function depends on two variables.   

Applications of these theorems will appear in a paper by the first author with
Hernandez, Parise, Payne, and Wang on a conjecture of Gromov
concerning the Almost Rigidity of the Scalar Torus Theorem \cite{AHMPPW1}.
The first author hopes to apply the techniques developed here in combination with
his prior work in \cite{Allen-Inverse} and \cite{Allen-Stability} to prove a special case
of Lee and the second author's conjecture on Almost Rigidity of the Positive Mass 
Theorem as stated in \cite{LeeSormani1}.   
Additional applications to conjectures involving scalar curvature that were raised by the second author at the Fields Institute and described in \cite{Sormani-scalar} will be explored with other teams of students and postdocs in the near future.  Anyone interested in joining one of these teams should contact the second author.

\vspace{.2in}
\noindent{\bf Acknowledgements:}   The authors would like to thank the Fields Institute and particularly Spyros Alexakis (University of Toronto), Walter Craig (McMaster University),
Robert Haslhofer (University of Toronto),
Spiro Karigiannis (University of Waterloo), Aaron Naber (Northwestern University),
McKenzie Wang (McMaster University) for organizing the Thematic Program and the  Summer School on Geometric Analysis there.   It provided a wonderful place for the two of us to work and meet with new people.  We'd like to thank  Christian Ketterer, Chen-Yun Lin, and Raquel Perales for serving as TAs to the students attending the second author's series of talks there.   Brian Allen would like to thank the United States Military Academy Department of Mathematics for funding his trip to join this team.  
Much of the work in this paper resulted from discussions there as to what was needed to complete the projects the teams were working on.  We wrote this paper to serve as a tool that could be applied by those teams as they meet again in the future.  All graphics in this paper were drawn by Penelope Chang of Hunter College High School, NYC.

\section{Review}

In this subsection we review what we mean by a warped product space even
with a noncontinuous warping function and what one needs to know about
Gromov-Hausdorff and Intrinsic Flat convergence to prove all examples and
theorems in this paper.   The reader does not need any prior knowledge of
these two notions of convergence.   Readers who are experts in these 
notions of convergence are recommended to read just the first and last subsections
of this review section of the paper, particularly Theorem~\ref{HLS-thm}
which combines results of Gromov in \cite{Gromov-metric} and the second author
with Huang and Lee in \cite{HLS}.   All examples and theorems in this paper
apply that theorem to prove convergence.

\subsection{Warped Product Spaces} \label{sect-warp} 

Let $(\Sigma^{n-1},\sigma)$ be a compact Riemannian manifold and 
\be
f:[r_1,r_2]\to \mathbb{R}^+
\ee
and define the warped product manifolds 
\be \label{MorN}
M=[r_1,r_2]\times_f \Sigma \textrm{ and } N={\mathbb S}^1\times_f \Sigma
\ee
with warped product metrics defined by
\be \label{eqn-g}
g= dr^2 + f^2(r) \sigma 
\ee 
where either $r\in [r_1,r_2]$ or $r\in {\mathbb S}^1$.   On such a manifold
we define lengths of curves to be
\be\label{eqn-L}
L_g(C) =\int_0^1 g(C'(t), C'(t))^{1/2} \, dt = \int_0^1 \sqrt{|r'(t)|^2 + |f(r(t))|^2|\theta'(t)|^2}\, dt
\ee
which is well defined even when $f$ is only $L^1$.  We then
define distances $d_g^M(p,q)$ and $d_g^N(p,q)$ on $M$ and $N$ respectively as
\be\label{eqn-d}
d_g(p,q) =\inf \{ L_g(C):\, C(0)=p, \, C(1)=q\}
\ee
where the value is different on $M$ and $N$ because the selection
of curves between points within these two spaces are different.

\begin{rmrk}\label{r-limit-geods}
Note that we do not need $f$ to be smooth or even continuous to define a warped product metric space.  As long as the function is bounded above,
we can define lengths using (\ref{eqn-L}).  Following the text of Burago-Burago-Ivanov\cite{BBI}, the distance $d$ defined by (\ref{eqn-d}) is symmetric and satisfies the triangle inequality.  It is positive definite as long as $f$ is bounded below by a positive number.  
Such a metric space is then compact and there are geodesics whose lengths achieve
the infimum in (\ref{eqn-d}).   
Even more general warped products of metric spaces
are explored by Alexander and Bishop in \cite{AB-warp}.  
\end{rmrk}

\begin{rmrk}\label{s-limit-geods} Throughout this paper we will assume that our
warping function $f$ is continuous.   Annegret Burtscher has proven that if a Riemannian manifold has a continuous metric tensor then the length of absolutely continuous curves defined by \eqref{eqn-L} is equivalent to the induced length defined by $d_{g}$ (See Definition 2.1, Proposition 4.1, and Theorem 4.11 of \cite{Burtscher-CtsRiemMetrics}). Hence if one considers $C_j(t)$ to be a sequence of absolutely continuous curves connecting $p,q \in M$ parameterized to be unit speed on $t \in [0,1]$ and so that
\begin{align}
L_g(C_j) \rightarrow d_g(p,q),
\end{align}
we can show that the distance is achieved by an absolutely continuous curve. First we can apply Theorem 2.5.14 of \cite{BBI} to conclude that since $L_g(C_j) \le L$ is uniformly bounded there exists a uniformly converging subsequence (where we just replace the original sequence with the subsequence) which converges to a curve of finite induced length $C_{\infty}$ so that $L_{d_g}(C_{\infty}) = d_g(p,q)$. We want to show that this curve is absolutely continuous so that $L_g(C_{\infty}) = d_g(p,q)$. To this end one notices that
\begin{align}
L_g(C_j) = \int_0^1 |C'_j(t)|_g dt < L
\end{align} 
and hence $|C'_j(t)|_g$ is a uniformly bounded family of $L^1$ functions.

By the constant speed parameterization we know
\begin{align}
d(C_j(a),C_j(b)) &\le L_g(C_j|_{[a,b]}) 
\\&= \int_a^b |C_j'(t)|_g dt \le C |b-a|, \quad 0 \le a < b \le 1,\label{equiintegrable}
\end{align}
which implies that $|C_j'(t)|_g$ is an equiintegrable sequence and hence there exists a $l_{\infty} \in L^1([0,1])$ so that  $|C_j'(t)|_g \rightharpoonup l_{\infty}$ in $L^1$.

 By considering the characteristic functions $\chi_{[a,b]}$ this implies
\begin{align}
\int_a^b |C'_j(t)|_g dt \rightarrow \int_a^b l_{\infty} dt, \quad 0 \le a < b \le 1.\label{WeakL1Convergence}
\end{align} 
By combining \eqref{equiintegrable} and \eqref{WeakL1Convergence} we find
\begin{align}
d(C_{\infty}(a),C_{\infty}(b)) \le \int_a^b l_{\infty} dt, \quad 0 \le a < b \le 1,
\end{align}
which is the definition of measure absolute continuity of a curve (See Definition 3.17 of \cite{Burtscher-CtsRiemMetrics}). Since this notion of absolute continuity agrees with the metric notion of absolute continuity (See Proposition 3.18 of \cite{Burtscher-CtsRiemMetrics}) we have shown that $C_{\infty}$ is an absolutely continuous curve which realizes the distance between $p$ and $q$.

The fact that the distance between points on a continuous Riemannian manifold is acieved by the length of an absolutely continuous curve will be important for us because we will repeatedly use the fact that the distance between points of $M$ can be achieved by an absolutely continuous curve $C(t)$ and hence we can reparameterize $C(t)$ so that $|C'(t)|_{g}=1$ almost everywhere.
\end{rmrk}

For warped products we can show that $L^2$ convergence of metrics $g_j \rightarrow g_{\infty}$ is equivalent to $L^2$ convergence of the warping functions $f_j \rightarrow f_{\infty}$. For this we fix the background metric $\delta = dr^2 + \sigma$ and an orthonormal basis for this metric $\{\partial_r,\partial_{\theta_1},...,\partial_{\theta_n}\}$ and compute
\begin{align}
\int_{M} |g_j-g_{\infty}|_{\delta}^2 dm &=  \int_{M}\sum_{i=1}^n |f_j-f_{\infty}|^2 \sigma(\partial_{\theta_i},\partial_{\theta_i})dm
\\&=n\int_{r_1}^{r_2}\int_{\Sigma} |f_j-f_{\infty}|^2 d\mu dr = n|\Sigma| \int_{r_1}^{r_2} |f_j-f_{\infty}|^2 dr,
\end{align}
where $dm$ is the measure on $M$ induced by $\delta$, $d\mu$ is the measure on $\Sigma$ from $\sigma$ and $|\Sigma|$ is $n$-dimensional volume of $\Sigma$. This shows that we can just work with $L^2$ convergence of the warping functions for the sake of this paper.

\subsection{Gromov-Hausdorff Convergence}

Gromov-Hausdorff convergence was introduced by Gromov in \cite{Gromov-metric}
See also the text of Burago-Burago-Ivanov\cite{BBI}.   It measures a distance
between metric spaces.   It is an intrinsic version of the Hausdorff distance between
sets in a common metric space $Z$:
\be 
d_H^Z(A_1, A_2) = \inf\{ r \, : \, A_1\subset T_r(A_2) \textrm{ and } A_2\subset T_r(A_1)\}
\ee 
where $T_r(A)=\{x \in Z: \, \exists a \in A \, s.t.\, d_Z(x,a)<r\}$.
Since an arbitrary given pair of compact metric spaces, $(X_i,d_i)$ might
not lie in the same compact metric space, we use distance preserving maps:
\be
\varphi_i: X_i \to Z \textrm{ such that } d_Z(\varphi_i(p), \varphi_i(q)) = d_i(p,q) \,\, \forall p,q \in X_i
\ee
to map them into a common compact metric space, $Z$.  

The Gromov-Hausdorff distance between two compact metric spaces, $(X_i,d_i)$,
is then defined to be
\be
d_{GH}((X_1,d_1),(X_2,d_2))=\inf\{ d_H^Z(\varphi_1(X_1), \varphi_2(X_2)) \, : \, \varphi_i: X_i\to Z\}
\ee
where the infimum is taken over all compact metric spaces $Z$ and all
distance preserving maps, $\varphi_i: X_i \to Z$.  

\subsection{Warped products as Integral Current Spaces}

Intrinsic flat convergence is defined for sequences of integral current spaces
by the second author jointly with Wenger in \cite{SW-JDG}.
An integral current space is a 
metric space, $(X,d)$, endowed with a current structure, $T$,
where $T$ is defined by a collection of biLipschitz charts with weights.
If we start with an oriented smooth Riemannian manifold, $M$, then $(X,d)$
is the standard metric space defined by $M$ using lengths of curves
as in (\ref{eqn-L}) and $T$ is defined by the orientation of $M$,
\be \label{currentT}
T(f, \pi_1,...,\pi_m) =  \int_{M} f \, d\pi_1 \wedge \cdots \wedge d\pi_m.
\ee

Here we are considering warped product spaces, $M$ and $N$, as in
(\ref{MorN}) allowing our function, $f:[r_1,r_2]\to \mathbb{R}^+$, to simply
have a maximum and a positive minimum and do not require it to be smooth.
In order to confirm that we still may use (\ref{currentT}) to define the integral
current structure on our space, we need only verify that our standard
oriented charts on the isometric product manifold are biLipschitz to the
metric $d$ we obtain as in (\ref{eqn-L})-(\ref{eqn-d}).   This is confirmed
by showing the identity map between the isometric product manifold,
$M_1=[r_1,r_2]\times_1 \Sigma$, and our warped product space,
$M=[r_1,r_2]\times_f \Sigma$, is biLipschitz:
 
\begin{lem} \label{biLip}
Suppose the warping function is bounded
\be
f(r) \in [a,b] \qquad \forall r\in [r_1,r_2],
\ee
then the identity map 
\be
F: M_1=[r_1,r_2]\times_1 \Sigma \,\,\,\to \,\,\,M=[r_1,r_2]\times_f \Sigma
\ee
 is
biLipschitz 
\be
0<\min \{a, 1\} \le \frac{d_M(F(p), F(q))}{d_{M_1}(p,q)}  \le (\max \{1,b\}).
\ee 
\end{lem}

\begin{proof}
This can be seen by observing that
\begin{eqnarray}
L_g(C) &=& \int_0^1 \sqrt{|r'(t)|^2 + |f(r(t))|^2|\theta'(t)|^2}\, dt  \\
&\le & (\max \{1,b\}) \int_0^1 \sqrt{|r'(t)|^2 + |\theta'(t)|^2}\, dt \\
&\le& (\max \{1,b\}) L_{g_1}(C).
\end{eqnarray}
Thus
\be
d_M(F(p), F(q)) \le (\max \{1,b\}) \,d_{M_1}(p,q)
\ee 
For the other direction we have
\begin{eqnarray}
L_{g_1}(C) &=& \int_0^1 \sqrt{|r'(t)|^2 +  |\theta'(t)|^2}\, dt  \\
&\le & (\min \{a, 1\})^{-1} \int_0^1 \sqrt{|r'(t)|^2 + |f(r(t))|^2 |\theta'(t)|^2}\, dt\\
& \le& 
(\min \{a, 1\})^{-1} L_g(C).
\end{eqnarray}
Thus
\be
d_{M_1}(p,q)\le (\min \{a, 1\})^{-1}
d_M(F(p), F(q)). 
\ee 
So we have our claim.
\end{proof}

\subsection{Key Theorem we apply to prove GH and $\mathcal{F}$ convergence}

The following theorem was proven by the second author jointly with 
Huang and Lee in \cite{HLS}
building upon earlier work of Gromov in \cite{Gromov-metric}.  
This theorem allows us to prove GH and intrinsic flat 
convergence using only information about the sequence of distance functions.
Note that it is a compactness theorem, providing the existence of a converging
subsequence once one simply has uniform biLipschitz control on the 
metrics.  The convergence is not biLipschitz convergence but instead it
is uniform convergence of the distance functions and also GH and $\mathcal{F}$
convergence of the spaces.  

\begin{thm}\label{HLS-thm}
Fix a precompact $n$-dimensional integral current space $(X, d_0, T)$
without boundary (e.g. $\partial T=0$) and fix
$\lambda>0$.   Suppose that
$d_j$ are metrics on $X$ such that
\be\label{d_j}
\lambda \ge \frac{d_j(p,q)}{d_0(p,q)} \ge \frac{1}{\lambda}.
\ee
Then there exists a subsequence, also denoted $d_j$,
and a length metric $d_\infty$ satisfying (\ref{d_j}) such that
$d_j$ converges uniformly to $d_\infty$
\be\label{epsj}
\epsilon_j= \sup\left\{|d_j(p,q)-d_\infty(p,q)|:\,\, p,q\in X\right\} \to 0.
\ee 
Furthermore
\be\label{GHjlim}
\lim_{j\to \infty} d_{GH}\left((X, d_j), (X, d_\infty)\right) =0
\ee
and
\be\label{Fjlim}
\lim_{j\to \infty} d_{\mathcal{F}}\left((X, d_j,T), (X, d_\infty,T)\right) =0.
\ee
In particular, $(X, d_\infty, T)$ is an integral current space
and $\set(T)=X$ so there are no disappearing sequences of
points $x_j\in (X, d_j)$.

In fact we have
\be\label{GHj}
d_{GH}\left((X, d_j), (X, d_\infty)\right) \le 2\epsilon_j
\ee
and 
\be\label{Fj}
d_{\mathcal{F}}\left((X, d_j, T), (X, d_\infty, T)\right) \le
2^{(n+1)/2} \lambda^{n+1} 2\epsilon_j \mass_{(X,d_0)}(T).
\ee
\end{thm}

\begin{rmrk}
In order to apply this theorem we will use the following method repeatedly.
We will demonstrate that a sequence has pointed convergence of the distance
functions and also satisfies the biLipschitz bound in (\ref{d_j}).  Then by this
theorem there is a converging subsequence.  However by the pointed convergence
we will see that all the subsequences must in fact converge to the same limit
space.  Thus we obtain $\mathcal{F}$ and GH convergence of the original sequence.
\end{rmrk}

\section{Examples}\label{Sect_Ex}
 
In this section we present our examples.  Each example contains a sequence 
of smooth warped product manifolds which converge in various ways to 
warped product metric spaces.  We first study distances on warped product spaces with
deep valleys.  We apply this to present our cinched warped product example.
We then observe what happens to distances on warped product spaces with
peaks.
 
\subsection{Distances on Warped Products with Valleys} 
 
First let us develop the intuitive picture first.  
Consider a warped product manifold $[-\pi,\pi]\times_g {\mathbb{S}}^1$
as in Figure~\ref{fig-cinch} with a warping function 
\be
 f_j(r)=
 \begin{cases}
 1 & r\in[-\pi,- 1/j]
 \\  h(jr) & r\in[- 1/j, 1/j]
 \\ 1 &r\in [1/j, \pi]
 \end{cases}
\ee
where $h$ is a smooth even function defining a valley with
$h(-1)=1$ with $h'(-1)=0$, 
decreasing to $h(0)=h_0\in (0,1]$ and then
increasing back up to $h(1)=1$, $h'(1)=0$.   Keep in mind that the distance
between the level sets, $r^{-1}(a)$ and $r^{-1}(b)$ is $|a-b|$ and so we have
evenly spaced levels drawn in the figure.  

\begin{figure} [h]
\centering
\includegraphics[width=4in]{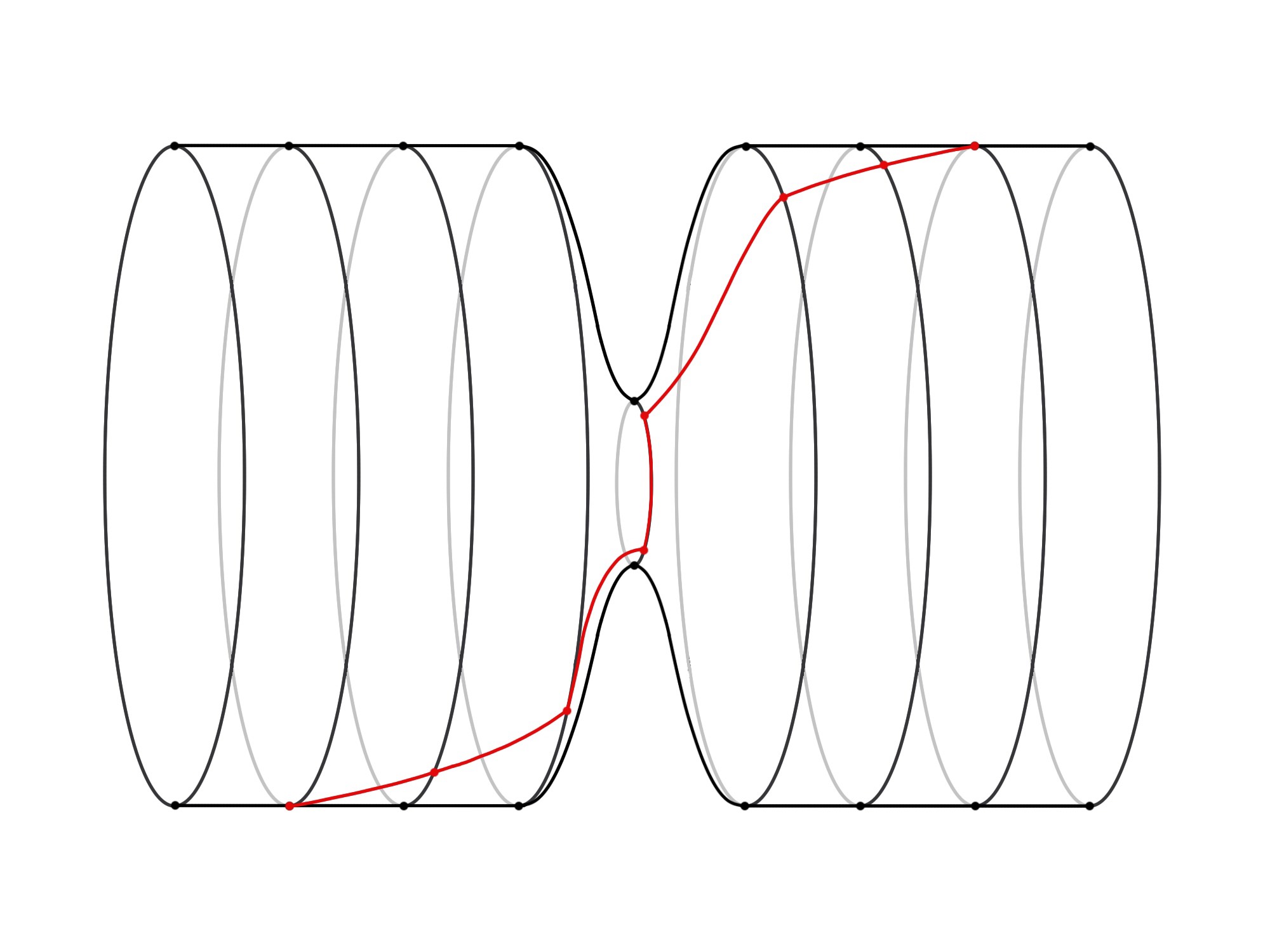}
\caption{The geodesic will cut across the valley}.
\label{fig-cinch}
\end{figure} 

A minimizing geodesic, draw in red in Figure~\ref{fig-cinch},
will proceed diagonally towards the valley, climb down into the valley, 
run along the valley, then climb out and proceed diagonally away from the valley.
The climbing parts are very short if the change in $r$ is small (which is true for
large $j$).  Since it is more efficient to travel around inside the valley (for the change in
$\theta$), it is more efficient to travel almost directly to the valley as in the geodesic
in the figure.  Observe that the length of this geodesic is bounded above by the
length of a curve which goes directly to the valley and straight down, then turns a right
angle to stay along the bottom of the valley, and then makes a right angle to climb out and
move directly to the end point.  Thus
\begin{eqnarray}
d((-r,\theta_1), (r,\theta_2)) &\le& |-r-0| + f(0)\,d_{{\mathbb{S}}^1}(\theta_1,\theta_2) + |0-r|.
\end{eqnarray}
In the following lemmas we use this same basic idea to bound distances in warped 
products with a
wide variety of warping functions.
 
\begin{lem}\label{min-level}   
Given a warped product space $M$ (or respectively $N$) defined as in (\ref{MorN}),
suppose $f(r) \ge f(r_0)$ for all $r\in [r_1,r_2]$ (or respectively $r \in {\mathbb S}^1$).
If $x_1, x_2 \in r^{-1}(r_0)$ then
\be
d_g(x_1,x_2)= f(r_0) d_\sigma(\theta_2,\theta_1).   
\ee
\end{lem}

\begin{proof}
Let $C(t)=(r(t), \theta(t))$ be any curve joining $x_1=(r_0, \theta_1)$ to 
$x_2=(r_0, \theta_2)$.   Then
\begin{eqnarray}
L(C[0,1]) &=& \int_0^1 \sqrt{|r'(t)|^2 + |f(r(t))|^2 |\theta'(t)|^2 } \, dt\\
&\ge &  \int_0^1 \sqrt{|0|^2 + |f(r_0)|^2 |\theta'(t)|^2 } \, dt\\
&=& f(r_0) \int_0^1  |\theta'(t)| \, dt\\
&=& f(r_0)\, L_{\Sigma}(\theta[0,1]) \\
&\ge& f(r_0) \,d_\sigma(\theta_2,\theta_1).
\end{eqnarray}
However if we take the curve $C(t)=(r_0, \theta(t))$ where $\theta(t)$
is a minimizing geodesic in $\Sigma$ from $\theta_1$ to $\theta_2$,
we have equality everywhere above.   So the infimum over all lengths
is achieved: 
\be
d_g(x_1,x_2)= \inf_C L(C[0,1])=f(r_0) d_\sigma(\theta_2,\theta_1).   
\ee
\end{proof}

\begin{lem}\label{distEst-L}
Given a warped product space $M$ defined as in (\ref{MorN})
and a pair of points $x_1 = (r_1,\theta_1)$ and $x_2=(r_2,\theta_2)$
with $r_1<r_2$
then the distance between those points is bounded by
\be
d^M_{g_j}(x_1,x_2) \le |r_2-r_1| + D_j(r_1,r_2)d_\sigma(\theta_2,\theta_1)\label{taxidistest}
\ee
where 
\be
\displaystyle D_j(r_1,r_2) = \min_{r \in [r_1,r_2]}f_j(r)
\ee
and $d_\sigma$ is the distance on $(\Sigma,\sigma)$.
\end{lem}

\begin{proof}
Let $\hat{r}_j \in (r_1,r_2)$ be chosen so that $f_j(\hat{r}_j)=D_j(r_1,r_2)$. 
Construct the following curve between the points $x_1, x_2 \in M_j$, where $\alpha \subset \Sigma$ is a geodesic with respect to $(\Sigma,\sigma)$, $\alpha(0) = \theta_1$ and $\alpha(1) = \theta_2$,
\be
C_j(t)=
\begin{cases}
(r_1+3(\hat{r}_j-r_1)t,\theta_1) & t \in [0,1/3]
\\ (\hat{r}_j, \alpha(3t-1)) & t \in [1/3,2/3]
\\ (\hat{r}_j+3(r_2-\hat{r})(t-2/3),\theta_2) & t \in [2/3,1]
\end{cases}
\ee
and then 
\be
d^M_{g_j}(x_1,x_2) \le L_j(C_j) = |r_2-\hat{r}_j|+f_j(\hat{r}_j)d_\sigma(\theta_2,\theta_1) 
+|\hat{r}_j-r_1|.
\ee
\end{proof}

Almost the same proof can be applied to show the following lemma:

\begin{lem}\label{distEst-L'}
Given a warped product space $N$ defined as in (\ref{MorN})
and a pair of points $x_1 = (r_1,\theta_1)$ and $x_2=(r_2,\theta_2)$
then the distance between those points is bounded by
\be
d^M_{g_j}(x_1,x_2) \le d_{{\mathbb S}^1}(r_1,r_2) + D_j(r_1,r_2)d_\sigma(\theta_2,\theta_1)\label{taxidistest'}
\ee
where 
\be
\displaystyle D_j(r_1,r_2) = \min_{r \in arc(r_1,r_2) }f_j(r)
\ee
where $arc(r_1,r_2)$ is the minor arc between $r_1$ and $r_2$ in ${\mathbb{S}}^1$
and where $d_\sigma$ is the distance on $(\Sigma,\sigma)$.
\end{lem}

\subsection{Cinched Spaces}

Here we see examples of spaces whose warping functions converge
in the $L^p$ sense but the $GH$ and $SWIF$ limits do not agree with
the $L^p$ limit due to the existence of deep canyons or cinching.
See Figure~\ref{fig-cinch} and now imagine that the valley remains equally as
deep but becomes very narrow.
 
 \begin{ex} \label{Cinched-Torus}  
 Consider the sequence of smooth functions $f_j(r):[-\pi,\pi]\to [1,2]$ 
 \be
 f_j(r)=
 \begin{cases}
 1 & r\in[-\pi,- 1/j]
 \\  h(jr) & r\in[- 1/j, 1/j]
 \\ 1 &r\in [1/j, \pi]
 \end{cases}
\ee
where $h$ is a smooth even function such that 
$h(-1)=1$ with $h'(-1)=0$, 
decreasing to $h(0)=h_0\in (0,1]$ and then
increasing back up to $h(1)=1$, $h'(1)=0$. 
Note that this defines a sequence of smooth Riemannian metrics, $g_j$,
as in (\ref{eqn-g}), with distances, $d_j$, as in (\ref{eqn-d})
on the manifolds, 
\be
M_j= [-\pi,\pi]\times_{f_j} \Sigma \textrm{ or } N_j={\mathbb{S}}^1\times_{f_j} \Sigma
\ee 
for any fixed Riemannian manifold $\Sigma$.   Consider also
$M_\infty$ and $N_\infty$ defined as above with $f_\infty(r)=1 \quad \forall r$.

Despite the fact that
\be
f_j \to f_\infty \textrm{ in } L^p \quad 
\ee
we do not have $M_j$ converging to $M_\infty$ nor $N_j$ to $N_\infty$
in the GH or $\mathcal{F}$ sense.  In fact
\be
M_j \GHto M_0 \textrm{ and } M_j \Fto M_0
\ee
and
\be
N_j \GHto N_0 \textrm{ and } N_j \Fto N_0
\ee
where $M_0$ and $N_0$ are warped metric spaces defined as in (\ref{MorN})
with warping factor
\be
 f_0(r)=
 \begin{cases}
 1 & r\in[-\pi, 0)
 \\  h_0  & r=0
 \\ 1 &r\in (0, \pi]
 \end{cases}
 .
\ee
\end{ex}

\begin{proof}
First we verify our claim about $L^p$ convergence  
\begin{align}
\left(\int_{-\pi}^{\pi}|f_j -1|^pdr \right)^{1/p} &= \left(\int_{-\tfrac{1}{j}}^{\tfrac{1}{j}} |h_j-1|^p dr\right)^{1/p}  \le \left(\frac{2}{j}\right)^{1/p} \rightarrow 0
\end{align}
where we use the fact that $|h_j-1|^p \le 1$ by construction.

Let us consider $(M_j, d_j)$.  Since we have
\be
0< h_0 \le f_j(r) \le f_0(r) \le f_\infty(r)=1
\ee
then
\be
(h_0)^2\, g_\infty \le g_j \le g_0 \le g_\infty
\ee
and
\be
h_0\, d_\infty(x_1,x_2)\le d_j(x_1,x_2) \le d_0(x_1,x_2) \le d_\infty(x_1,x_2).
\ee
Using $d_\infty$ as our background metric we can apply the theorem
in the appendix of \cite{HLS} to see that a subsequence of the $d_j$ converges
uniformly to some limit, $d$, such that
\be \label{dled0}
h_0 \,d_\infty(x_1,x_2)\le d(x_1,x_2) \le d_0(x_1,x_2) \le d_\infty(x_1,x_2).
\ee
In addition the subsequences
converge in the Gromov-Hausdorff and Intrinsic Flat sense:
\be
(M_j, d_j) \GHto (M,d) \textrm{ and }(M_j, d_j, T) \Fto (M,d,T).
\ee
We need only prove $d=d_0$ for then no subsequence was necessary and we
have proven our example.

Consider $x_1, x_2\in M$ such that
\be
d(x_1, x_2) < \min\{ d(x_1,p) + d(p, x_2) \,:\, p\in r^{-1}(0)\}.
\ee
So there exists $\delta>0$ depending on these two points such that
\be
d(x_1, x_2) +\delta \le  \min\{ d(x_1,p) + d(p, x_2) \,:\, p\in r^{-1}(0)\}.
\ee
Then for $N$ sufficiently large, and all $j\ge N$ (in our subsequence) we have
\be
d_j(x_1, x_2) +\delta/2 \le \min\{ d_j(x_1,p) + d_j(p, x_2) \,:\, p\in r^{-1}(0)\}.
\ee
Thus the $L_{g_j}$-shortest curve, $\gamma_j$, between $x_1$ and $x_2$
avoids $r^{-1}(-\delta/4,\delta/4)$.  Here we have $g_j=g_0=g_\infty$ 
so its length
is the same with respect to all three metrics:
\be
L_{g_j}(\gamma_j)=L_{g_0}(\gamma_j)=L_{g_\infty}(\gamma_j).
\ee
So
\be
d_j(x_1,x_2) \ge d_0(x_1,x_2) 
\ee
and taking the limit we have 
\be
d(x_1,x_2)\ge d_0(x_1, x_2)
\ee and 
combining this with (\ref{dled0}) we have
\be
d(x_1,x_2)= d_0(x_1, x_2).
\ee 
In fact for any $L_d$-shortest curve $\gamma$, 
\be\label{segment}
\gamma([t_1,t_2])\cap r^{-1}(0) = \emptyset \implies 
d(\gamma(t_1),\gamma(t_2))=d_0(\gamma(t_1),\gamma(t_2)).
\ee

We need only confirm that $d(x_1,x_2)= d_0(x_1, x_2)$ for
$x_1, x_2\in M$ such that
\be
d(x_1, x_2) = \min\{ d(x_1,p) + d(p, x_2) \,:\, p\in r^{-1}(0)\}.
\ee
Taking the $L_d$-shortest curve $\gamma$ between $x_1$ and $x_2$,
we know that $s_1\le s_2$
\be
s_1=\inf\{t: \, \gamma(t) \in r^{-1}(0)\} 
\ee
and 
\be
s_2=\sup\{t: \, \gamma(t) \in r^{-1}(0)\}.
\ee
We have
\be
d(x_1, x_2)=L_d(\gamma)=d(\gamma(0), \gamma(s_1))
+d(\gamma(s_1), \gamma(s_2))+d(\gamma(s_2), \gamma(1))
\ee
By (\ref{segment}) if $s_1>0$ then for all $\delta>0$ we have
\be
d(\gamma(0), \gamma(s_1-\delta))=d_0(\gamma(0), \gamma(s_1-\delta))
\ee
so
\be
d(\gamma(0), \gamma(s_1))=d_0(\gamma(0), \gamma(s_1)).
\ee
Similarly
\be
d(\gamma(s_2), \gamma(1))=d_0(\gamma(s_2), \gamma(1)).
\ee

Thus we need only confirm that $d(x_1,x_2)= d_0(x_1, x_2)$ for
$x_1, x_2\in r^{-1}(0)$.  This easily follows by applying
Lemma~\ref{min-level} 
to both $f_j$ and $f_0$ since
both functions have minimum $= h_0$ at $r=0$:
\begin{eqnarray}
d(x_1,x_2) 
&=& \lim_{j\to \infty} d_j(x_1, x_2)\\
&= &  h_0 d_\sigma(\theta_1,\theta_2)\\
&=& d_0(x_1,x_2).
\end{eqnarray}

To prove the case where we have a warped product of the form $N$
as in (\ref{MorN}) the proof is almost the same.
\end{proof} 
 
\subsection{Moving Cinches}
 
 Here we explore what happens when the warping functions
 converge in $L^p$ but not pointwise almost everywhere.
 
 \begin{ex}\label{MovingCinched}
We first construct a classical sequence of smooth functions
$f_j:[-\pi,\pi]\to (0,1]$ which converge $L^p$ to $f_\infty=1$ but do not
converge pointwise almost everywhere without taking a subsequence.
 Let
  \be
 f_j(r)=
 \begin{cases}
  h((r-t_j)/\delta_j ) & r\in [t_j-\delta_j, t_j +\delta_j] 
 \\ 1 & \textrm{ elsewhere }
 \end{cases}
\ee
where $h$ is a smooth even function as in Example~\ref{Cinched-Torus}
such that 
$h(-1)=1$ with $h'(-1)=0$, 
decreasing to $h(0)=h_0\in (0,1]$ and then
increasing back up to $h(1)=1$, $h'(1)=0$, and where
 \be
 \{t_j:\, j\in \mathbb{N}\}=\left\{\tfrac{0}{1}, \tfrac{1}{1}, \tfrac{0}{2}, \tfrac{1}{2}, \tfrac{2}{2}, 
 \tfrac{0}{4}, \tfrac{1}{4}, \tfrac{2}{4}, \tfrac{3}{4}, ...\right\}
 \ee
 and
 \be
 \{\delta_j:\, j \in \mathbb{N}\} =\left\{\tfrac{1}{1}, \tfrac{1}{1}, \tfrac{1}{2}, \tfrac{1}{2}, \tfrac{1}{2}, 
 \tfrac{1}{4}, \tfrac{1}{4}, \tfrac{1}{4}, \tfrac{1}{4}...\right\}.
 \ee
Then the cylinders, $N_j$, defined as in (\ref{MorN}) will not converge in the 
GH or ${\mathcal{F}}$ sense without taking a subsequence.  The tori $M_j$
will converge since each torus in this sequence is isometric
to a torus in the sequence of tori in Example~\ref{Cinched-Torus} via an isometry
which moves $t_j$ to $0$.
 \end{ex}
 
 \begin{proof}
First we check that $f_j$ converges in $L^p$ but not pointwise almost everywhere. To this end we check that
\begin{align}
\left (\int_{-\pi}^{\pi}|f_j-1|^p dr \right) ^{1/p}&=\left (\int_{t_j-\delta_j}^{t_j+\delta_j}|h_0-1|^p dr \right) ^{1/p} = (2\delta_j)^{1/p} \rightarrow 0
\end{align}
since $|h_0-1|^p \le 1$ by construction. Of course we do not find pointwise convergence for any $r \in [0,1]$ since for every choice of $J >0$ one can find a $j_1 \ge J$ and a $r \in [-\pi,\pi]$ so that $f_{j_1}(r) = h_0$ and another $j_2 \ge J$ so that $f_{j_2}(r) =1$.  

Now if we take a subsequence where $t_{j_k}=0$, then exactly as in 
Example~\ref{Cinched-Torus} we see that $N_{j_k}$ converges in the GH and $\mathcal{F}$ sense
to $N_0$ of that example.   On the other hand, if we take a subsequence where 
$t_{j'_k}=1$, then imitating the proof in 
Example~\ref{Cinched-Torus} we see that $N_{j'_k}$ converges in the GH and $\mathcal{F}$ sense
to $N_0'$ which is a warped product whose warping function is $1$ everywhere except
at $r=1$ where it is $h_0$.   Thus the original sequence of $N_j$ of this example
has no GH nor $\mathcal{F}$ limit.
\end{proof}

\subsection{Avoiding Ridges}
  
The cinched spaces of Example~\ref{Cinched-Torus} 
 did not converge to their $L^p$ limit because their warping functions, $f_j$, all had a minimum uniformly below the level of their $L^p$ limit, $f_\infty$.  
 Here we will see there is no corresponding problem when the $f_j$
 have a maximum uniformly above the level of their $L^p$ limit.   
 
 \begin{figure} [h]
\centering
\includegraphics[width=4in]{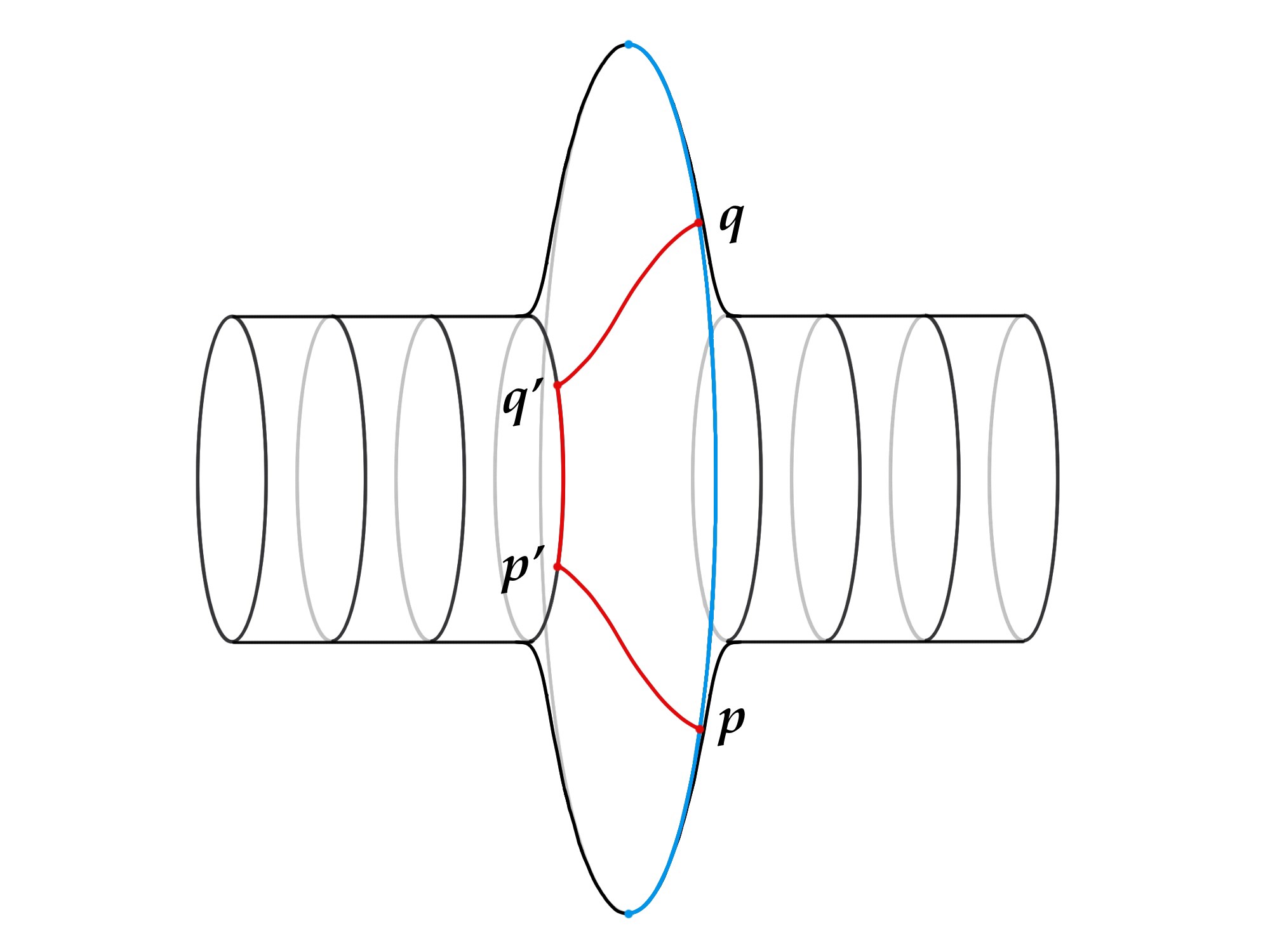}
\caption{A curve $\gamma$ from $p$ to $q$ on a ridge, which
first cuts down to $p'$ and then runs across to $q'$ before cutting up
to $q$ is shorter than curve running along the ridge between $p$ and $q$
if the rudge is narrow enough.}
\label{fig-ridge}
\end{figure} 

 In the following lemma, 
 we have a ridge as in Figure~\ref{fig-ridge}, the minimal geodesic between
 points, $p,q$ lying on that ridge, will not run along the ridge.  In the following 
 we consider $f_j$ with a
 maximum at $r_*$ and thus there is a ridge along the level set $f_j^{-1}(r_*)$.
 
 \begin{lem}\label{AvoidPeak2}
 Given $r_*, \hat{r} \in [r_0,r_1]$, the distance between
 $x_1=(r_*,\theta_1)$ and $x_2=(r_*,\theta_2)$ 
 in a warped product space 
 is bounded above
 by
 \be
d(x_1,x_2)\le 2|\hat{r}-r_*| + f_j(\hat{r}) d_\sigma(\theta_1,\theta_2).
\ee
Thus for a fixed $r_* \in [r_0,r_1]$, 
  if there exists a $\hat{r} \in [r_0,r_1]$ so that
 \be
f_j(\hat{r}) < f_j(r_*) - 2\frac{|\hat{r}-r_*|}{d_\sigma(\theta_1,\theta_2)}\label{L-cond}
\ee
then the minimizing geodesic from $x_1=(r_*,\theta_1)$ to $x_2=(r_*,\theta_2)$,$\theta_1,\theta_2\in \Sigma$, $\theta_1 \not = \theta_2$, cannot be a curve with constant $r$-component, $r(t)=r_*$.
 \end{lem}
 
See Figure~\ref{fig-ridge} taking
$p=x_1=(r_*,\theta_1)$ and $q=x_2=(r_*,\theta_2)$ 
and $p'=(\hat{r},\theta_1)$ and $q=x_2=(\hat{r},\theta_2)$.
So $d(p,q) \le L(\gamma)=d(p,p')+d(p',q') +d(q',q)$ where
$d(p,p')=d(q,q')=|r_*-\hat{r}|.$
 
 \begin{proof}
Let $x_1,x_2 \in M_j$ with coordinates $x_1=(r_*,\theta_1)$ to $x_2=(r_*,\theta_2)$,$\theta_1,\theta_2\in \Sigma$, $\theta_1 \not = \theta_2$ so that \eqref{L-cond} is satisfied for $r_*$. Let $\alpha\subset \Sigma$ be a curve between $\theta_1,\theta_2$ with length $L_{\Sigma}(\alpha)=d_\sigma(\theta_1,\theta_2)$ and consider the curve
\be
\gamma(t)=
\begin{cases}
(r_*+3(\hat{r}-r_*)t, \theta_1) & t \in [0,1/3]
\\ (\hat{r}, \alpha(3t-1)) & t \in [1/3,2/3]
\\ (\hat{r}+3(r_*-\hat{r})(t-2/3),\theta_2) & t \in [2/3,1]
\end{cases}
\ee
as depicted in Figure~\ref{fig-ridge}.  Then 
\be
L_j(\gamma) = 2|\hat{r}-r_*| + f_j(\hat{r}) d_\sigma(\theta_1,\theta_2).
\ee
So if we consider $\beta(t) = (r_*,\alpha(t))$ 
and use the assumption \eqref{L-cond} then we find that
\be
L_j(\gamma) < L_j(\beta)
\ee
and hence $\beta(t)$ cannot be the minimizing geodesic.
 \end{proof}

\subsection{A Single Ridge Disappears}

Here we see that a sequence of warped product spaces with a consistently
high ridge that is increasingly narrow converges in the $L^p$,
ptwise a.e., $GH$, and $\mathcal{F}$
sense to an isometric product manifold as if the ridge simply disappears
despite the fact that the
warping functions do not converge pointwise to the constant function $1$. 
See Figure~\ref{fig-ridge}.

\begin{ex}\label{SingleRidge}
Consider the sequence of functions $f_j(r):[-\pi,\pi]\to [1,2]$ with
 \be
 f_j(r)=
 \begin{cases}
 1 & r\in[-\pi,- 1/j]
 \\  h(jr) & r\in[- 1/j, 1/j]
 \\ 1 &r\in [1/j, \pi]
 \end{cases}
\ee
where $h=h_{ridge}$ is a smooth even function such that 
$h(-1)=1$ with $h'(-1)=0$, 
increasing to $h(0)=h_0\in (1,2]$ and then
decreasing back down to $h(1)=1$, $h'(1)=0$. 
Note that this defines a sequence of smooth Riemannian metrics, $g_j$,
as in (\ref{eqn-g}), with distances, $d_j$, as in (\ref{eqn-d})
on the manifolds, 
\be
M_j= [-\pi,\pi]\times_{f_j} \Sigma \textrm{ or } N_j={\mathbb{S}}^1\times_{f_j} \Sigma
\ee 
for any fixed Riemannian manifold $\Sigma$.   Consider also
$M_\infty$ and $N_\infty$ defined as above with $f_\infty(r)=1 \quad \forall r$.
Here we have
\be
f_j \to f_\infty=1 \textrm{ in } L^p \textrm{ but not ptwise }
\ee
and yet $M_j \to M_\infty$ and $N_j \to N_\infty$ in both the GH and $\mathcal{F}$
sense.
\end{ex}

\begin{proof}
First we check that $f_j$ converges in $L^p$  to $f_\infty$. To this end we check that
\begin{align}
\left (\int_{-\pi}^{\pi}|f_j-f_\infty|^p dr \right) ^{1/p}&=\left (\int_{-1/j}^{1/j}|h(jr)-1|^p dr \right) ^{1/p} 
\le (2/j)^{1/p} \rightarrow 0
\end{align}
since $|h_j-1|^p \le 1$ by construction.   Observe that $f_j$ does not converge
pointwise to $f_\infty$ because $f_j(0)=h_0 > 1=f_\infty (0)$.
Let
\be \label{Jdeltaridge}
J_\delta= 1/\delta
\ee
so that $f_j(r)=f_\infty(r)$ on $[0, -1/j]\cup [1/j, 1]$ for all $j\ge J_\delta$.

Next observe that since $2 f_\infty(r) \ge f_j(r) \ge f_\infty(r)$ at all $r\in [-\pi,\pi]$ we have
\be \label{for-HLS-ridge}
d_\infty(p,q) \le d_j(p,q) \le 2 d_\infty(p,q) \qquad \forall p,q.
\ee

Since our limit space, $M_\infty$, is an isometric product space, 
any pair of points $x_1=(s_1,\theta_1)$ to $x_2=(s_2,\theta_2)$ 
with $s_1<s_2$ is
joined by a smooth $L_\infty$ minimizing 
geodesic, $C:[0,1]\to M_\infty$, such that
\be
d_\infty(p,q)=L_\infty(C).
\ee
In fact $C(t)=(r(t), \theta(t))$ where $r:[0,1] \to [r_1,r_2]$ is strictly increasing from
from $s_1$ to $s_2$,
and $\theta:[0,1] \to \Sigma$ is a geodesic from $\theta_1$ to $\theta_2$ with respect to $(\Sigma,\sigma)$.
Let $T_\delta \subset [0,1]$ be defined as the possibly empty interval
\be
T_\delta =\{t:\, r(t) \in [-\delta,\delta]\}.
\ee
Observe that the length of $C$ restricted to the interval $T_\delta$ satisfies
\be
L_\infty(C(T_\delta))\le 2 \delta L_\infty(C) \le 2 \delta d_\infty(x_1,x_2).
\ee
For $j\ge J_\delta$ as in (\ref{Jdeltaridge}), we have
\begin{eqnarray}
d_j(x_1,x_2) \le L_j(C) &=& \int_0^1 g_j(C'(t),C'(t))^{1/2} \, dt \\
&\le & \int_{T_\delta} 2 g_\infty(C'(t),C'(t))^{1/2}\\
&& + \int_{[0,1]\setminus T_\delta}  g_\infty(C'(t),C'(t))^{1/2}\\
&\le &  2 L_\infty(C(T_\delta)) + L_\infty(C[0,1])\\
&\le & (1 + 2 \delta) d_\infty(x_1,x_2).
\end{eqnarray} 
Thus for $x_1$ and $x_2$ lying on different levels of $r$ we have pointwise convergence
$d_j(x_1,x_2) \to d_\infty(x_1,x_2)$.

Taking points that lie on the same level, $x_1=(s,\theta_1)$ to $x_2=(s,\theta_2)$,
we know that the minimizing geodesic, $C$, in our isometric product will have the
form $C(t)=(s, \theta(t))$.  If the points do not lie on the ridge, $s\neq 0$, and so
\be
d_j(x_1,x_2) \le L_j(C)=L_\infty(C) =d_\infty(x_1,x_2) \qquad \forall j \ge J_\delta.
\ee 
So again we have pointwise convergence $d_j(x_1,x_2) \to d_\infty(x_1,x_2)$.

If the points both lie on the ridge $x_1=(0,\theta_1)$ to $x_2=(0,\theta_2)$
then by Lemma~\ref{AvoidPeak2} we have 
\begin{eqnarray} 
d_j(x_1,x_2) &\le& 1 d_\Sigma(\theta_1,\theta_2) + 2\delta \qquad \forall j \ge J_\delta\\ &=& d_\infty(x_1,x_2) + 2 \delta \qquad \forall j \ge J_\delta. 
\end{eqnarray}
And again we have pointwise convergence $d_j(x_1,x_2) \to d_\infty(x_1,x_2)$.

By Theorem~\ref{HLS-thm} combined with (\ref{for-HLS-ridge}) we know a subsequence
$d_{j_k}$ converges uniformly to some limit distance.  Since we have pointwise 
convergence to $d_\infty$, we know in fact that $d_j$ thus converge uniformly to $d_\infty$
without even taking a subsequence.  Furthermore we have Gromov-Hausdorff and intrinsic flat convergence.

The proof when we have warped around ${\mathbb S}^1$ to create $N_j$ is very similar.
\end{proof}

\subsection{Moving Ridges}
 
 Here we see a sequence of spaces which have $f_j$ converging
 to $f_\infty=1$ in the $L^p$ sense and $f_j \ge 1$.  The 
 sequence does not converge pointwise almost everywhere unless
 one takes a subsequence.  Nevertheless by Theorem~\ref{WarpConv}
 there is a GH and a SWIF limit without taking a subsequence
 and indeed the limit is the space warped by $f_\infty$. 
 
 \begin{ex}\label{MovingRidges} 
We first construct a classical sequence of smooth functions
$f_j:[-\pi,\pi]\to [1,2]$ which converge $L^p$ to $f_\infty=1$ but do not
converge pointwise almost everywhere without taking a subsequence.
Let
 \be
 \{s_j:\, j\in \mathbb{N}\}=\left\{\tfrac{0}{1}, \tfrac{1}{1}, \tfrac{0}{2}, \tfrac{1}{2}, \tfrac{2}{2}, 
 \tfrac{0}{4}, \tfrac{1}{4}, \tfrac{2}{4}, \tfrac{3}{4}, ...\right\}
 \ee
 and
 \be
 \{\delta_j:\, j \in \mathbb{N}\} =\left\{\tfrac{1}{1}, \tfrac{1}{1}, \tfrac{1}{2}, \tfrac{1}{2}, \tfrac{1}{2}, 
 \tfrac{1}{4}, \tfrac{1}{4}, \tfrac{1}{4}, \tfrac{1}{4}...\right\}.
 \ee
 Let
  \be
 f_j(r)=
 \begin{cases}
  h((r-s_j)/\delta_j ) & r\in [s_j-\delta_j, s_j +\delta_j] 
 \\ 1 & \textrm{ elsewhere }
 \end{cases}
\ee
where $h$ is a smooth even function such that 
$h(-1)=1$ with $h'(-1)=0$, 
increasing up to $h(0)=h_0\in (1,2]$ and then
decreasing back down to $h(1)=1$, $h'(1)=0$. 
Note that this defines a sequence of smooth Riemannian metrics, $g_j$,
as in (\ref{eqn-g}), with distances, $d_j$, as in (\ref{eqn-d})
on the manifolds, 
\be
M_j= [-\pi,\pi]\times_{f_j} \Sigma \textrm{ or } N_j={\mathbb{S}}^1\times_{f_j} \Sigma
\ee 
for any fixed Riemannian manifold $\Sigma$. Consider also
$M_\infty$ and $N_\infty$ defined as above with $f_\infty(r)=1 \quad \forall r$.
Here we have
\be
f_j \to f_\infty=1 \textrm{ in } L^p \textrm{ but not ptwise }
\ee
and yet $M_j \to M_\infty$ and $N_j \to N_\infty$ in both the GH and $\mathcal{F}$
sense.
 \end{ex}
 
 \begin{proof}
First we check that $f_j$ converges in $L^p$ but not pointwise almost everywhere. To this end we check that
\begin{align}
\left (\int_{-\pi}^{\pi}|f_j-1|^p dr \right) ^{1/p}&=\left (\int_{s_j-\delta_j}^{s_j+\delta_j}|h_j-1|^p dr \right) ^{1/p} = (2\delta_j)^{1/p} \rightarrow 0
\end{align}
since $|h_j-1|^p \le 1$ by construction. Of course we do not find pointwise convergence for any $r \in [-\pi,\pi]$ since for every choice of $J >0$ one can find a $j_1 \ge J$ so that $f_{j_1}(r) = 0$ and another $j_2 \ge J$ so that $f_{j_2}(r) > 0$.

The proof of the Gromov-Hausdorff and Intrinsic Flat convergence follows almost
exactly as in Example~\ref{SingleRidge} except that we must choose $J_\delta$
and $T_\delta$ differently.  We skip this proof since the convergence follows from 
Theorem~\ref{WarpConv} anyway.
\end{proof}
 
\subsection{Many Ridges}
 
 Here we see a sequence of spaces which have $f_j$ converging
 to $f_\infty=1$ in the $L^p$ sense and $f_j \ge 1$.  The 
 sequence converges pointwise to a nowhere continuous function.
 Nevertheless by Theorem~\ref{WarpConv}
 there is a GH and a SWIF limit without taking a subsequence
 and indeed the limit is the isometric product space.
 
 \begin{ex}\label{ManyRidges} 
We first construct a classical sequence of smooth functions
$f_j:[-\pi,\pi]\to [1,2]$ as in Figure~\ref{fig-many-ridges}
which converge $L^p$ to $f_\infty=1$ but do not
converge pointwise almost everywhere without taking a subsequence.
Let
 \begin{eqnarray}
 S&=&\left\{s_{i,j}=-\pi + \tfrac{2\pi i}{2^j}\,: \,  i=1,2,... (2^j-1),\, j\in \mathbb{N}\right\}\\
&=& \left\{-\pi + \tfrac{2\pi}{2},-\pi+\tfrac{2\pi}{4}, -\pi + \tfrac{2\pi 2}{4}, -\pi + \tfrac{2\pi3}{4}, 
-\pi + \tfrac{2\pi}{8},...\right\}
 \end{eqnarray}
 which is dense in $[-\pi,\pi]$
 and
 \be
 \{\delta_j=(1/2)^{2j}:\, j \in \mathbb{N}\} =\{1/4, 1/16, 1/32,....\}.
 \ee
 Let
  \be
 f_j(r)=
 \begin{cases}
  h((r-s_{i,j})/\delta_j ) & r\in [s_{i,j}-\delta_j, s_{i,j} +\delta_j] \textrm{ for } i =1..2^j-1
 \\ 1 & \textrm{ elsewhere }
 \end{cases}
\ee
where $h$ is a smooth even function such that 
$h(-1)=1$ with $h'(-1)=0$, 
increasing up to $h(0)=h_0\in (1,2]$ and then
decreasing back down to $h(1)=1$ with $h'(1)=0$. 
Note that this defines a sequence of smooth Riemannian metrics, $g_j$,
as in (\ref{eqn-g}), with distances, $d_j$, as in (\ref{eqn-d})
on the manifolds, 
\be
M_j= [-\pi,\pi]\times_{f_j} \Sigma \textrm{ or } N_j={\mathbb{S}}^1\times_{f_j} \Sigma
\ee 
for any fixed Riemannian manifold $\Sigma$. 
Consider also
$M_\infty$ and $N_\infty$ defined as above with $f_\infty(r)=1 \quad \forall r$.
Here we have
\be
f_j \to f_\infty=1 \textrm{ in } L^p \textrm{ but not ptwise }
\ee
and yet $M_j \to M_\infty$ and $N_j \to N_\infty$ in both the GH and $\mathcal{F}$
sense.
 \end{ex}
 
 \begin{figure} [h]
\centering
\includegraphics[width=5in]{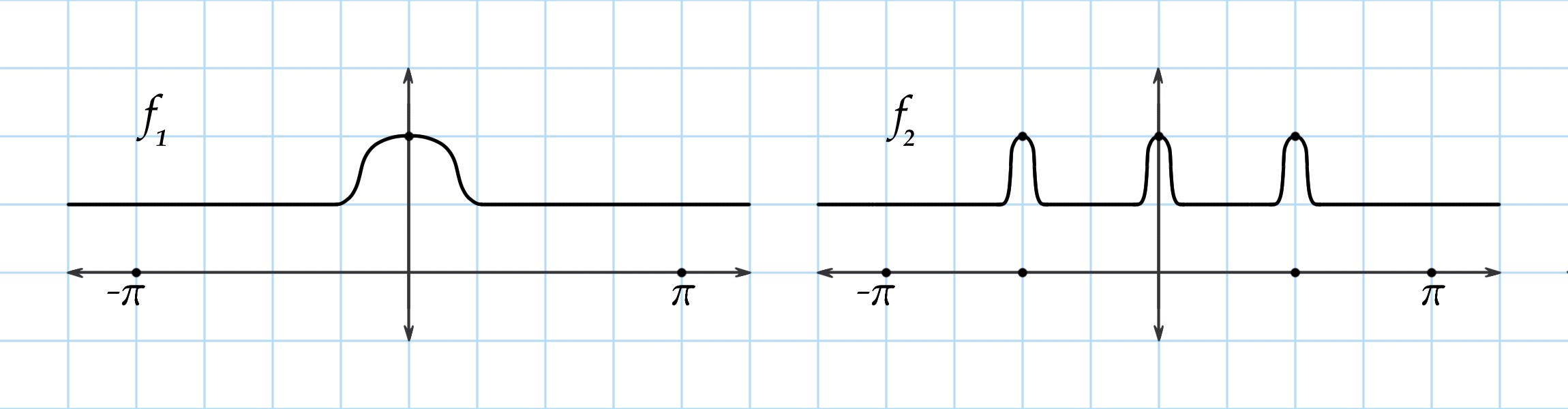}
\caption{The warping functions of Example~\ref{ManyRidges}.}
\label{fig-many-ridges}
\end{figure}
 
 \begin{proof}
First we check that $f_j$ converges in $L^p$ 
\begin{eqnarray}
\left (\int_{-\pi}^{\pi}|f_j-1|^p dr \right) ^{1/p}
&=& \left(\sum_{i=1}^{2^j-1} \int_{s_{i,j}-\delta_j}^{s_{i,j}+\delta_j}|f_j-1|^p dr \right) ^{1/p} \\
&=& \left((2^j-1)(2\delta_j)\right)^{1/p} \\
&=& \left((2^j-1)(1/2)^{2j}\right)^{1/p}\rightarrow 0.
\end{eqnarray}
Next observe that $f_j$ converges pointwise on $S$ to $h_0$
and pointwise to $1$ elsewhere.  Since $S$ is dense and
$h_0>1$ the pointwise limit is continuous nowhere.

The proof of the Gromov-Hausdorff and Intrinsic Flat convergence follows almost
exactly as in Example~\ref{SingleRidge} except that we must choose $J_\delta$
and $T_\delta$ differently.  We skip this proof since the convergence follows from 
Theorem~\ref{WarpConv} anyway.
\end{proof}

\subsection{Converging to Euclidean-Taxi Spaces}   

In Theorem~\ref{WarpConv} we will prove that if $f_j \ge 1$ and $f_j \to 1$ in the
$L^p$ sense then we have Gromov-Hausdorff and Intrinsic Flat convergence to
the isometric product space just as in Examples~\ref{SingleRidge}, ~\ref{MovingRidges}
and~\ref{ManyRidges}.   We now investigate what might happen if 
$f_j$ does not converge to $1$ in the $L^p$
sense but does have a dense collection of points where $f_j$ converges pointwise to
$1$.   In the example below we see that this does not suffice to prove GH or intrinsic flat
convergence to the isometric product space.   

Here we will construct a sequence of warped product spaces
with increasingly many cinches.  The limit metric we obtain in this example
is not a Riemannian metric but a metric of the following form:

\begin{defn}\label{defn-ET}
Let $M$ and $N$ be product manifolds as in (\ref{MorN}).  For any $R>1$, we
define the minimized R-stretched Euclidean taxi metric ($R-ET$ metric)
between $x_1=(s_1,\theta_1)$ and $x_2=(s_2,\theta_2)$ to be
\be
d^M_{R-ET}(x_1,x_2)= \min_{\Theta \in [0,d_\Sigma(\theta_1, \theta_2)]} \sqrt{|s_1-s_2|^2 + R^2\Theta^2} + d_\Sigma(\theta_1, \theta_2)-\Theta.
\ee
\be
d^N_{R-ET}(x_1,x_2)= \min_{\Theta \in [0,d_\Sigma(\theta_1, \theta_2)]} \sqrt{d_{{\mathbb S}^1}(s_1,s_2)^2 + R^2\Theta^2} + d_\Sigma(\theta_1, \theta_2)-\Theta.
\ee
\end{defn}

Note that the $R-ET$ metric is smaller than the isometric product metric with the $\theta$
direction scaled by $R$ (achieved at $\Theta=d_\Sigma(\theta_1, \theta_2)$), and 
it is also smaller than the taxi product (achieved at $\Theta=0$).   One may view the
$R-ET$ metric as an infimum over lengths of all curves which are partly line segments 
of the form $\theta=ms+\theta_0$ (whose lengths are measured by stretching the Euclidean
metric by $R$ in the $\theta$ direction) and partly vertical segments purely in the $\theta$ direction (whose lengths are not rescaled).  Without stretching, taking $R=1$, we
see the minimum is achieved going purely diagonal with the standard Euclidean metric.

It is not immediately obvious that $R-ET$ metrics are true metrics satisfying positivity, symmetry and the triangle inequality.   We prove this in the following lemma:

\begin{lem} \label{ET}
When 
\be \label{near-ET}
d_\Sigma(\theta_1, \theta_2)\,\,\,\le \,\,\,\frac{|s_1-s_2|}{R\sqrt{R^2-1}} 
\ee
then the metric is an isometric product
\be \label{near-ET-2}
d^M_{R-ET}((s_1, \theta_1),(s_2,\theta_2))= \sqrt{|s_1-s_2|^2 + R^2 d_{\Sigma}(\theta_1, \theta_2)^2}. 
\ee 
and otherwise the metric is a stretched taxi product:
\be\label{near-ET-3}
d^M_{R-ET}((s_1, \theta_1),(s_2,\theta_2))= |s_1-s_2|\left( \frac{\sqrt{R^2-1}}{R} \right)+ d_{\Sigma}(\theta_1, \theta_2). 
\ee
In fact $d^M_{R-ET}$ is a minimum of these two metrics and is a length metric
whose balls are the unions of diamonds and ellipses as in Figure~\ref{R-ET-balls}.
It is a true metric satisfying positivity, symmetry and the triangle inequality.   
\end{lem}

\begin{figure} [h]
\centering
\includegraphics[width=3in]{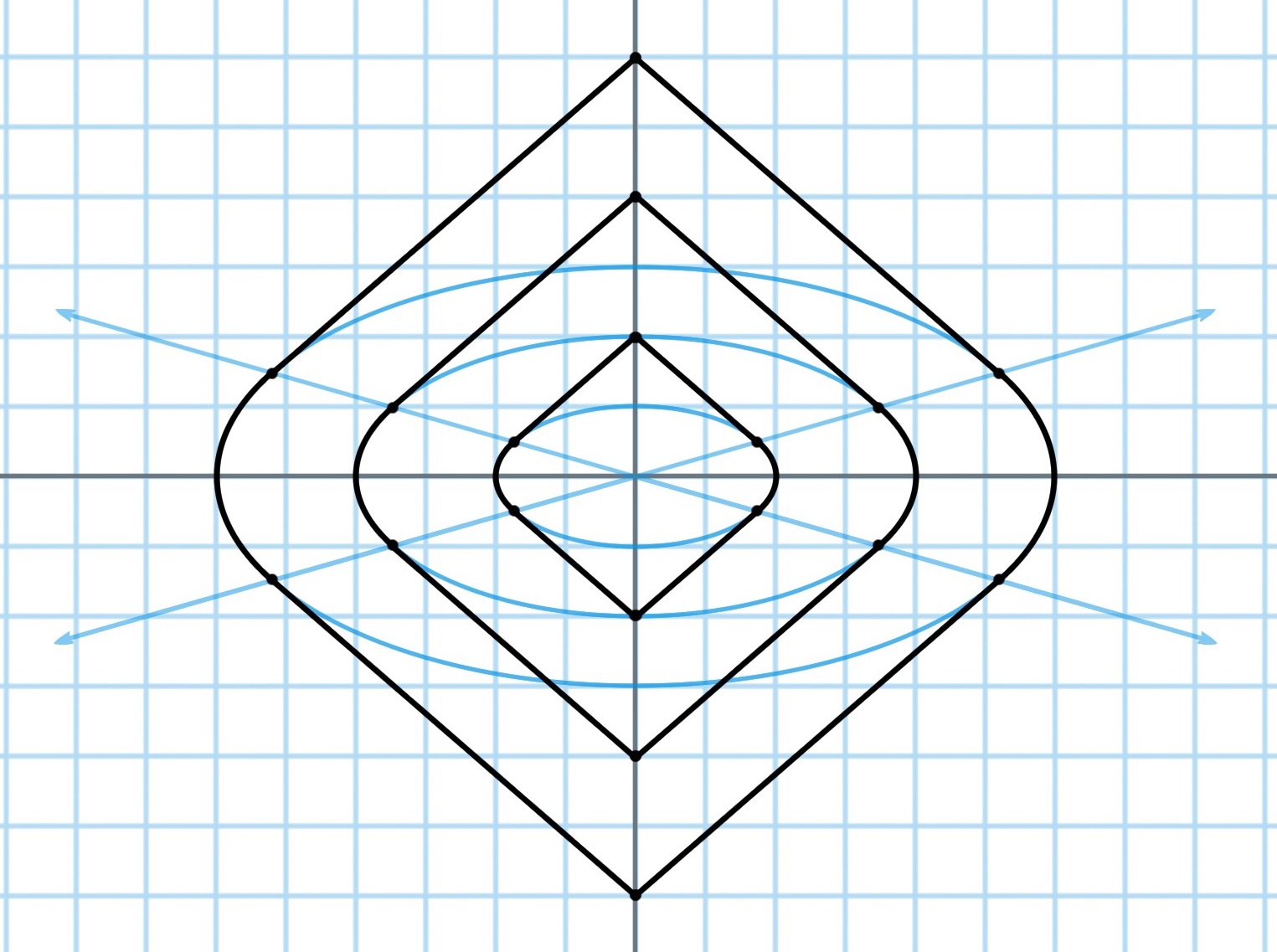}
\caption{The concentric balls of radius $r=2$, $4$, and $6$ in an $R-ET$ space with
$R=2$ are unions of diamonds, $|s|+\tfrac{\sqrt{3}}{2}|\theta|<r$, and ellipses,
$s^2+2\theta^2<r^2$.  }
\label{R-ET-balls}
\end{figure} 

\begin{proof}
To locate the minimum in the definition of the ET metric, we take the derivative 
\be
\frac{d}{d\Theta} \sqrt{|s_1-s_2|^2 + R^2\Theta^2} + d_\Sigma(\theta_1, \theta_2)-\Theta=
\qquad \qquad\qquad\qquad
\ee
\be
=(1/2) (|s_1-s_2|^2 + R^2\Theta^2)^{-1/2}(2R^2 \Theta) -1.
\ee
This derivative is negative at $\Theta=0$ so the minimum
is not achieved by the taxi product metric.  The derivative becomes $0$
at 
\be \label{Theta-0}
\Theta_0=\frac{|s_1-s_2|}{R\sqrt{R^2-1}}
\ee
and is then positive for $\Theta>\Theta_0$.   If (\ref{near-ET}) holds then 
$\Theta_0$ does not lie in $(0,d_\Sigma(\theta_1, \theta_2))$, so the
minimum is achieved at 
$
\Theta=d_\Sigma(\theta_1, \theta_2)
$
and we have (\ref{near-ET-2}).

Otherwise, the minimum is achieved at $\Theta_0$.  Since
\be
R^2\Theta_0^2= |s_1-s_2|^2 / (R^2-1) \textrm{ and } 1+ (1/(R^2-1))= R^2/(R^2-1)
\ee
we have
\begin{eqnarray}
\qquad d^M_{R-ET}((s_1, \theta_1),(s_2,\theta_2))&\le&  \sqrt{|s_1-s_2|^2 + R^2\Theta_0^2} + d_\Sigma(\theta_1, \theta_2)-\Theta_0 \\
&=&\frac{|s_1-s_2|\cdot |R|}{ \sqrt{R^2-1} } + d_\Sigma(\theta_1, \theta_2) 
     - \frac{|s_1-s_2|}{R\sqrt{R^2-1}}  \\
&=&\frac{|s_1-s_2| (R^2-1)}{ R\sqrt{R^2-1} }     + d_\Sigma(\theta_1, \theta_2) \\
&=& |s_1-s_2| \frac{\sqrt{R^2-1}}{R} + d_\Sigma(\theta_1, \theta_2). 
\end{eqnarray}
Thus we have (\ref{near-ET-3}).  

We also see that $d^M_{R-ET}((s_1, \theta_1),(s_2,\theta_2))$ is the minimum
of the two metrics in (\ref{near-ET-2}) and (\ref{near-ET-3}).
We know that both these metrics are length metrics. Indeed the metric in
(\ref{near-ET-2}) is the infimum of the lengths of curves, $C(t)=(s(t), \theta(t))$
where
\be
L_E(C) = \int_0^1 \sqrt{ s'(t)^2 + R^2 g_\Sigma(\theta'(t), \theta'(t))} \, dt
\ee 
and the metric in
(\ref{near-ET-3}) is the infimum of the lengths of curves, $C(t)=(s(t), \theta(t))$
where
\be\label{near-ET-4}
L_T(C)= \int_0^1 |s'(t)| \frac{\sqrt{R^2-1}}{|R|} + g_\Sigma(\theta'(t), \theta'(t))^{1/2} \, dt.
\ee
Thus
\be
d^M_{R-ET}(x_1, x_2) = \min\{ \inf_C L_E(C), \inf_C L_T(C) \}
= \inf_C L_{R-ET}(C)
\ee
where $L_{R-ET}(C)=\min\{L_E(C), L_T(C) \}$.  Thus we have positivity and symmetry
(which was easy to see) and now the triangle inequality as well (which was not).
\end{proof}

We now present our example: a sequence of warped product spaces with
increasingly many cinches which converges in the uniform, GH and $\mathcal{F}$
sense to a produce space with a minimized R-stretched Euclidean taxi metric.
Here we have $R=5$, but we could easily construct similar sequences converging to
any $R-ET$ metric with $R>1$.

\begin{ex} \label{to-R-ET}
Let
\begin{eqnarray}
 S&=&\left\{s_{i,j}=-\pi + \tfrac{2\pi i}{2^j}\,: \,  i=1,2,... (2^j-1),\, j\in \mathbb{N}\right\}\\
&=& \left\{-\pi + \tfrac{2\pi}{2},-\pi+\tfrac{2\pi}{4}, -\pi + \tfrac{2\pi 2}{4}, -\pi + \tfrac{2\pi3}{4}, 
-\pi + \tfrac{2\pi}{8},...\right\}
 \end{eqnarray}
 which is dense in $[-\pi,\pi]$
 and
 \be
 \{\delta_j=(1/2)^{2j}:\, j \in \mathbb{N}\} =\{1/4, 1/16, 1/32,....\} 
 \ee
 Define the functions $f_j$ as in Figure~\ref{fig-to-R-ET} as follows
  \be
 f_j(r)=
 \begin{cases}
  h((r-s_{i,j})/\delta_j ) & r\in [s_{i,j}-\delta_j, s_{i,j} +\delta_j] \textrm{ for } i =1..2^j-1
 \\ 5 & \textrm{ elsewhere }
 \end{cases}
\ee
where $h$ is a smooth even function such that 
$h(-1)=5$ with $h'(-1)=0$, 
decreasing down to $h(0)=1$ and then
increasing back up to $h(1)=5$ with $h'(1)=0$. 

Then $f_j(r)\ge 1$ converges pointwise
to 1 on the dense set, $S$.

If we define $M_j$ and $N_j$ as in (\ref{MorN}) 
then they do not converge to isometric products with warping function $1$.
Instead they converge in the GH and $\mathcal{F}$ sense 
to a product manifold with a $R-ET$ metric with $R=5$.
\end{ex}

\begin{figure} [h]
\centering
\includegraphics[width=5in]{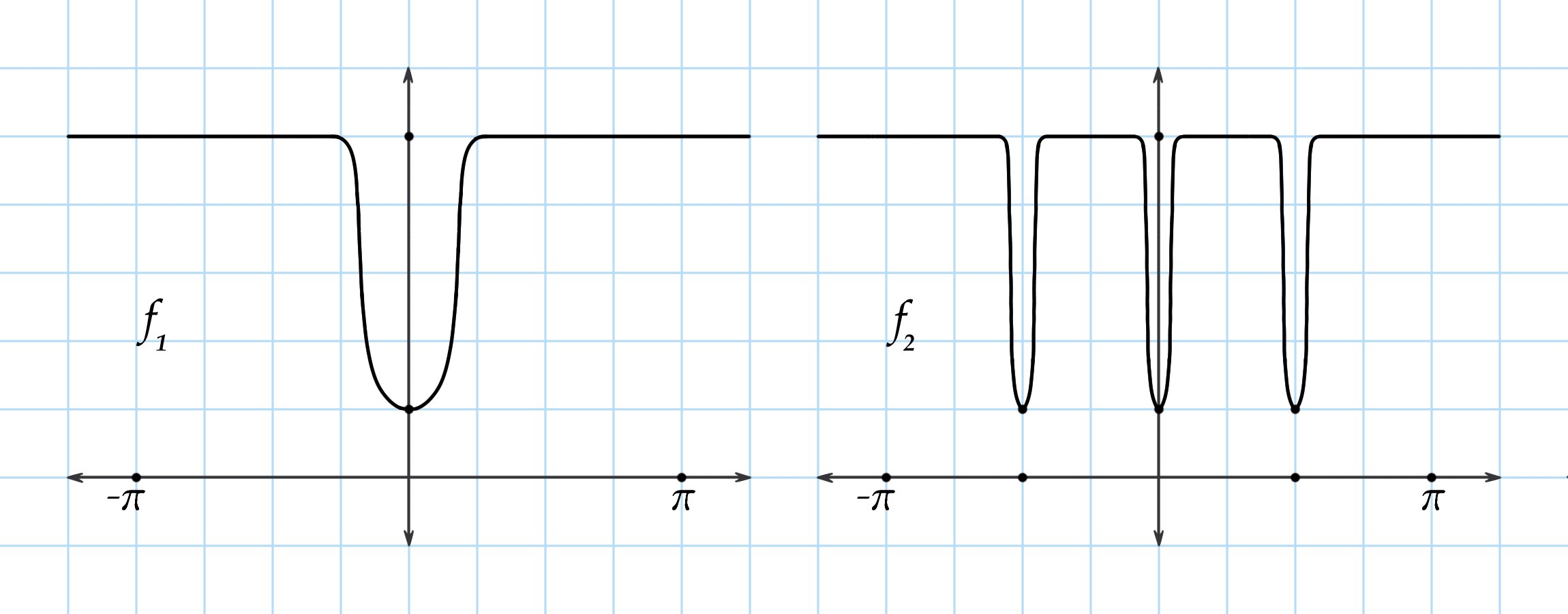}
\caption{The warping functions of Example~\ref{to-R-ET}.}
\label{fig-to-R-ET}
\end{figure}

\begin{proof}
First we check that $f_j \rightarrow 5$ in $L^p$ by using the fact that $|f_j-5|^p \le 4^p$ 
\begin{eqnarray}
\left (\int_{-\pi}^{\pi}|f_j-5|^p dr \right) ^{1/p}
&=& \left(\sum_{i=1}^{2^j-1} \int_{s_{i,j}-\delta_j}^{s_{i,j}+\delta_j}|f_j-5|^p dr \right) ^{1/p} \\
&\le& \left((2^j-1)(2\delta_j)4^p\right)^{1/p} \\
&=& 4\left((2^j-1)(1/2)^{2j}\right)^{1/p}\rightarrow 0.
\end{eqnarray}

Now observe that since 
\begin{align}\label{WarpIneq}
1 \le f_j(r) \le 5 \qquad  \forall r\in [-\pi,\pi]
\end{align} 
we have
\begin{eqnarray}\label{for-HLS-3j}
d_1(p,q) \le d_j(p,q) \le 5 d_1(p,q),
\end{eqnarray}
where $d_1$ is the warped product metric with warping function $1$.
Thus by \cite{HLS},
a subsequence of the warped product manifolds converges in
the uniform, GH and intrinsic flat sense to some limit metric space
with limit metric $d_\infty$
\be \label{for-HLS-3infty}
d_1(p,q) \le d_\infty(p,q) \le 5 d_1(p,q) \qquad \forall p,q.
\ee
  We will show that the pointwise limit of
the $d_j$ is $d_{5-ET}$, thus
proving that the original sequence of warped product manifolds converges
in the uniform, GH and intrinsic flat sense to the Euclidean/taxi space.

Let us consider an arbitrary pair of points, $x_i=(s_i, \theta_i)$.  
If $\theta_1=\theta_2$ then
\be
d_j(x_1, x_2) = |s_1-s_2|= d_{5-ET}(s_1, s_2).
\ee

In general, if $\theta_1\neq \theta_2$ let $s'_{i,j} \in f_j^{-1}(1)$ with 
\be\label{s'ij}
|s'_{i,j}-s_i|<2\pi/2^j
\ee
\be
x'_{i,j}=(s'_{i,j}, \theta_i). 
\ee  
By the triangle inequality applied two ways we have
\begin{eqnarray}\label{use-twice}
|d_j(x_1,x_2) - d_j(x_{1,j}', x_{2,j}')|&\le& d_j(x_1, x'_{1,j})+ d_j(x_{2,j}', x_2)\\
&\le & |s_1-s'_{1,j}| + |s'_{2,j}-s_2|  < 4\pi/2^j
\end{eqnarray}
and 
\begin{eqnarray}\label{use-twice-taxi}
& & |d_{5-ET}(x_1,x_2) - d_{5-ET}(x_{1,j}', x_{2,j}')|
\\&\qquad & \le d_{5-ET}(x_1, x'_{1,j})+ d_{5-ET}(x_{2,j}', x_2)\\
&\qquad & \le  |s_1-s'_{1,j}| + |s'_{2,j}-s_2|  < 4\pi/2^j  
\end{eqnarray}

Recall that to complete the proof we must prove the pointwise limit:
\be
\lim_{j\to \infty} d_j(x_1, x_2) = d_{5-ET}(x_1,x_2).
\ee
By (\ref{use-twice}) we need only show
\be
\lim_{j\to \infty} d_j(x_{1,j}', x_{2,j}') = d_{5-ET}(x_1,x_2).
\ee

Applying the triangle inequality again, with $x_{1,\theta,j}=(s'_{1,j}, \theta)$ where $\theta \in \Sigma$ so that $d_{\Sigma}(\theta_2,\theta) \in [0,d_{\Sigma}(\theta_1,\theta_2)]$,
we have
\begin{eqnarray}
d_j(x_{1,j}', x_{2,j}')&\le& d_j(x_{1,j}', x_{1,\theta,j})+ d_j(x_{1,\theta,j}, x_{2,j}')\\
&\le & d_\Sigma(\theta_1, \theta) + \sqrt{|s_{1,j}'-s_{2,j}'|^2 + 25d_{\Sigma}(\theta_2,\theta)^2},
\end{eqnarray}
where we have used \eqref{WarpIneq} in the last line.
Since this is true for any $\theta \in \Sigma$ so that $d_{\Sigma}(\theta_2,\theta) \in [0,d_{\Sigma}(\theta_1,\theta_2)]$ we find
\begin{eqnarray}
d_j(x_{1,j}', x_{2,j}')&\le& d_{5-ET}(x'_{1,j},x'_{2,j}).
\end{eqnarray}
Thus taking the limsup and applying (\ref{use-twice-taxi}) we have
\be \label{limsupdtaxi}
\limsup_{j\to \infty} d_j(x_{1,j}', x_{2,j}')
\le \limsup_{j\to \infty} d_{5-ET}(x'_{1,j},x'_{2,j})=d_{5-ET}(x_1,x_2).
\ee

So now we need only show
\be\label{NeedOnlyShow}
\liminf_{j\to \infty} d_j(x_{1,j}', x_{2,j}') \ge d_{5-ET}(x_1,x_2).
\ee
By (\ref{use-twice-taxi}) we need only show 
\be \label{NeedOnlyShow2}
\liminf_{j\to \infty} \left(d_j(x_{1,j}', x_{2,j}')- d_{5-ET}(x_{1,j}',x_{2,j}') \right)\ge 0. 
\ee

If $s'_{1,j}=s'_{2,j}$ then
\begin{eqnarray}\label{s'1j=s'2j}
d_j(x_{1,j}', x_{2,j}')&\ge & d_\Sigma(\theta_1, \theta_2) =d_{5-ET}(x'_{1,j},x'_{2,j}).
\end{eqnarray}

If $s_{1,j}'\neq s_{2,j}'$, then the $L_j$ shortest path, $C_j(t)=(r(t), \theta(t))$, 
from $x_{1,j}'$ to $x_{2,j}'$ must pass
from one valley over to the other, possibly passing through many valleys in
between.  Observe that
\be\label{two-parts}
d_j(x_{1,j}', x_{2,j}')=L_j(C_j) = L_j( C_j \cap f^{-1}(5)) + L_j(C_j \setminus f^{-1}(5)).
\ee

The segments of $C_j$ which intersect $f_j^{-1}(5)$ lie in an
product space warped by the constant function $5$ so
\be\label{RjThetaj}
L_j( C_j \cap f^{-1}(5))= \sqrt{R_j^2 +25 \Theta_j^2}
\ee
where $R_j$ is the sum of changes in $r$ on these segments and
where $\Theta_j$ is the sum of distances in $\Sigma$ between the theta values of
the endpoints of these segments.   

Let $R_0=|s_1-s_2|$ which is the total change in $r$ along $C_j$.   By the definition of $\delta_j$, 
\be
2^j \delta_j=2^j (1/2^{2j}) \to 0.
\ee
Since we have at most $2^j$ intervals where $f_j<5$,
we see that as 
\be\label{RjtoR0}
\lim_{j\to \infty} R_0-R_j = 0.
\ee
So the total change in $r$ for the segments in $C_j \setminus f^{-1}(5)$ is
converging to $0$.  

Let $\Theta_0=d_{\Sigma}(\theta_1, \theta_2)$.  Then $\Theta_0-\Theta_j$ is
the sum of distances in $\Sigma$ between the theta values of
the endpoints of the segments in $C_j \setminus f^{-1}(5)$.  Since the
warping factors $f_j(r)\ge 1$ everywhere, the distance between the 
endpoints of each segment is $\ge$ distance in $\Sigma$ between the theta values of
the endpoints of the segment.  Thus 
\be \label{Theta-0j}
 L_j(C_j \setminus f^{-1}(5))\ge \Theta_0-\Theta_j.   
 \ee

Combining this together with (\ref{two-parts}) and (\ref{RjThetaj})we have
\begin{align} \label{sqrt-diff}
d_j(x_{1,j}', x_{2,j}')=L_j(C_j) &\ge \sqrt{R_j^2 + 25 \Theta_j^2} + \Theta_0-\Theta_j 
\\&\ge \inf_{\Theta \in [0,d_{\Sigma}(\theta_1,\theta_2)]} \sqrt{R_j^2 + 25 \Theta^2} + \Theta_0-\Theta 
\end{align}
Since
\begin{align} \label{LimtToET}
\lim_{j \rightarrow \infty}\left( \inf_{\Theta \in [0,d_{\Sigma}(\theta_1,\theta_2)]} \sqrt{R_j^2 + 25 \Theta^2} + \Theta_0-\Theta \right )= \lim_{j \rightarrow \infty}d_{5-ET}(x_{1,j}', x_{2,j}')
\end{align}
 we are done by combining \eqref{sqrt-diff} and \eqref{LimtToET} which shows \eqref{NeedOnlyShow2}.
\end{proof}

\begin{rmrk}\label{RemScalarCompactness}
 If we take the isometric product of Example \ref{to-R-ET} with a standard circle, $\bar{N}_j^3 = N_j^2 \times \mathbb{S}^1$, $\Sigma=\mathbb{S}^1$, then we have a sequence of 3-manifolds satisfying all the hypotheses of the scalar compactness conjecture of Gromov-Sormani \cite{IAS} (recently proved in the rotationally symmetric case by Park-Tian-Wang \cite{PTW})
 \begin{align}
Vol(\bar{N}_j) &\le 5 Vol(\mathbb{T}^3)
\\ diam(\bar{N}_j) &\le 5 diam(\mathbb{T}^3)
\\minA(\bar{N}_j) &\ge minA(\mathbb{T}^3)
\end{align}
except for the the scalar curvature bound. Therefore, this example demonstrates that the conclusion of the scalar compactness conjecture, that the SWIF limit have Euclidean tangent cones almost everywhere, requires the scalar curvature bound. We note that the volume and diameter bound follow since $f_j \le 5$ and the minA bound follows since $f_j \ge 1$.
\end{rmrk}

\section{Proof of the Main Theorem}\label{sect-WarProd}

The goal of this section is to prove our main theorem, Theorem~\ref{WarpConv}. 

In this theorem,
$M_j=[r_0,r_1]\times_{f_j} \Sigma$ where $\Sigma$ is an $n-1$ dimensional manifold including also $M_j$ without boundary that have $f_j$ periodic with period $r_1-r_0$ as in \eqref{MorN}.  We assume that the warping factors, $f_j\in C^0([r_0,r_1])$, satisfy the following:
  \be
  0 < f_{\infty} - \frac{1}{j} \le f_j(r) \le K 
  \ee
and
  \be
  f_j(r) \rightarrow f_{\infty}(r) \textrm{ in }L^2
  \ee
  where $f_\infty\in C^0([r_0,r_1])$.

The proof of Theorem~\ref{WarpConv} proceeds as follows.  In Lemma \ref{distLowerBound} we use the $C^0$ lower bound to show that 
\be
\liminf_{j \rightarrow \infty} d_{j}(p,q) \ge d_{\infty}(p,q) \textrm{ pointwise.}
\ee 
We use the $L^2$ convergence of $f_j \rightarrow f_{\infty}$ in Lemma \ref{warp-L} and Lemma \ref{ConstRLengthBound}, combined with the estimate of Lemma \ref{Theta(C_j)-WEst}, to show that the lengths of fixed curves with respect to $M_j$  and $M_{\infty}$ converge. We apply this result to a fixed geodesic with respect to $g_{\infty}$, to prove that
\be 
\limsup_{j \rightarrow \infty} d_{j}(p,q) \le d_{\infty}(p,q) \textrm{ pointwise.}
\ee 
Thus in Proposition~\ref{prop-ptwise} we have the pointwise limit
\be
 \lim_{j \rightarrow \infty} d_j(p,q) = d_{\infty}(p,q).
 \ee
 To complete the proof of uniform, GH and SWIF convergence using Theorem \ref{HLS-thm}, as is done in the examples in section \ref{Sect_Ex}, we need uniform bounds
 on $d_j$ proven in Lemma~\ref{biLip}.

\subsection{Assuming a $C^0$ lower bound}

We have seen in Section~\ref{Sect_Ex} that in order to get Gromov-Hausdorff convergence to agree with $L^2$ convergence we will need a $C^0$ lower bound on $f_j$ and so now we see the consequence of this assumption for the distance between points. 

\begin{lem}\label{distLowerBound} Let $p,q \in [r_0,r_1]\times \Sigma$ and assume that $f_{j}(r) \ge f_{\infty} - \frac{1}{j} > 0$, $diam(M_j) \le D$ then
\begin{align}
\liminf_{j \rightarrow \infty} d_{j}(p,q) \ge d_{\infty}(p,q)
\end{align}
and furthermore we find the uniform estimate
 \be
  d_{g_j}(p,q)-d_{g_{\infty}}(p,q) \ge  -\frac{\sqrt{2}\max_{[r_0,r_1]}\sqrt{f_{\infty}} D}{\min_{[r_0,r_1]}f_{j}(r)\sqrt{j}}.
\ee
\end{lem}

\begin{proof}
Let $C_j(t)=(r_j(t),\theta_j(t))$ be the absolutely continuous curve in $M_j$, parameterized so that $|C_j|_{g_j}=1$ a.e., realizing the distance between $p$ and $q$.  Then compute
\begin{align}
 d_{g_j}(p,q) &= \int_0^{L_j(C_j)} \sqrt{r_j(t)^2+f_j(r_j(t))^2 |\theta_j'(t)|^2}dt
 \\&\ge \int_0^{L_j(C_j)} \sqrt{r_j(t)^2+(f_{\infty}(r_j(t)) - \tfrac{1}{j})^2 |\theta_j'(t)|^2}dt
  \\&= \int_0^{L_j(C_j)} \sqrt{r_j(t)^2+f_{\infty}(r_j(t))^2|\theta_j'(t)|^2 -  \left(\tfrac{2}{j}f_{\infty}(r_j(t))|\theta_j'(t)|^2 - \tfrac{1}{j^2} |\theta_j'(t)|^2\right )}dt\label{posIntegrand}
  \end{align}
  Now we use the inequality $\sqrt{|a-b|}\ge |\sqrt{a}-\sqrt{b}|\ge \sqrt{a}-\sqrt{b}$ in succession, employing the fact that the integrand in \eqref{posIntegrand} is positive and the square roots that follow are of positive quantities by the assumptions of the lemma.
  \begin{align}
   d_{g_j}(p,q)&\ge \int_0^{L_j(C_j)}\left| \sqrt{r_j(t)^2+f_{\infty}(r_j(t))^2|\theta_j'(t)|^2} - \tfrac{1}{\sqrt{j}}|\theta_j'(t)|\sqrt{\left(2f_{\infty}(r_j(t)) - \tfrac{1}{j} \right )}\right|dt
  \\&\ge \int_0^{L_j(C_j)} \sqrt{r_j(t)^2+f_{\infty}(r_j(t))^2|\theta_j'(t)|^2}dt 
  \\&\qquad- \tfrac{1}{\sqrt{j}}\int_0^{L_j(C_j)}|\theta_j'(t)|\sqrt{\left(2f_{\infty}(r_j(t)) - \tfrac{1}{j} \right )}dt
  \\&\ge L_{g_{\infty}}(C_j)-\tfrac{1}{\sqrt{j}}\int_0^{L_j(C_j)}|\theta_j'(t)|\sqrt{\left(2f_{\infty}(r_j(t)) - \tfrac{1}{j} \right )}dt
 \end{align} 
 
Now we notice that 
\begin{align}
&\sqrt{f'_j(t)^2+f_j(r_j(t))^2|\theta_j'(t)|^2} = 1 \text{ a.e.}
\\&\Rightarrow |\theta_j'(t)| \le \frac{1}{\min f_j} \text{ a.e.}
\end{align}
 which allows us to compute
 \be
  d_{g_j}(p,q) \ge d_{g_{\infty}}(p,q) -\frac{\sqrt{2}\max_{[r_0,r_1]}\sqrt{f_{\infty}} D}{\min_{[r_0,r_1]}f_{j}(r)\sqrt{j}},
\ee
where the diameter bound from the hypotheses is used to conclude that $L_j(C_j) \le D$. The desired result follows by taking limits.
 
\end{proof}

\subsection{$L^2$ convergence and convergence of lengths}
In this section we would like to observe the consequence of $L^2$ convergence of $f_j \rightarrow f_{\infty}$ for convergence of lengths of curves and distances between points
in $M_j$ culminating in an estimate on the pointwise limsup of the 
distance functions [Proposition~\ref{prop-limsup}]. 

We start by proving we have uniform bounds on the diameter:

 \begin{lem} \label{lem-diam} 
 If $\|f_j-f_{\infty}\|_{L^2} \le \delta_j$ and $M_j$ are warped products as in \ref{MorN}  
 then
 \be
   \diam(M_j) \le 2|r_1-r_0| 
  + \left( \|f_\infty\|_{C_0} + \tfrac{\delta_j}{\sqrt{r_1-r_0\,} }\right) \diam(\Sigma)
  \ee
  \end{lem}
  
  \begin{proof}
  Let $p,q\in M_j$.   Recall that the distance between these
  points is the infimum over lengths of all curves.  For any $r\in [r_0,r_1]$ we can take a 
  first path from $p$ radially to the level $r$, then a second path around that level $r$,
  and then a third path from that level to $q$.  The first and third paths each have
   length $\le |r_1-r_0|$, and the middle path has length bounded above by the diameter of the
   level.  Thus we have
  \begin{eqnarray}
 \qquad  d_j(p,q) &\le& 2|r_1-r_0| + f_j(r)\diam(\Sigma)  \\
  &\le & 2|r_1-r_0| + \left(\, f_\infty(r) + |f_j(r)-f_\infty(r)|\, \right) \,\diam(\Sigma).
  \end{eqnarray}
  Choosing an $r$ such that 
  \be
  |f_j(r)-f_\infty(r)|^2 \le \frac{1}{r_1-r_0} \int |f_j(s)-f_\infty(s)|^2 \, ds
  \ee
  we have
  \be
  |f_j(r)-f_\infty(r)| \le \frac{\|f_j-f_\infty\|_{L_2}}{\sqrt{r_1-r_0\,}} 
  \ee
  and  $f_\infty(r) \le \|f_\infty\|_{C_0}$.
  \end{proof}
  
  Recall that in warped product manifolds with continuous warping functions
  we have absolutely continuous curves whose length achieves the distance between
  two points [Remark~\ref{s-limit-geods}].   
  
We next consider the length of a fixed curve which is monotone in $r$. 

\begin{lem} \label{warp-L}
Fix an absolutely continuous curve $C(t)=(r(t), \theta(t))$, $t \in [0,1]$, which is monotone in $r$.
If $\|f_j-f_{\infty}\|_{L^2} \le \delta=\delta_j$ 
and $M_j$ are warped products as in \eqref{MorN}
then
\be
|L_j(C)-L_\infty(C)| 
\le \left(\delta^2 + 4\|f_\infty\|_{L^2}^2\right) \delta^{1/2} \Theta(C)
\ee
where
\be\label{eq-Theta}
 \Theta(C)=\left( \int_{r(0)}^{r(1)} |\theta'(r)|^2\, dr \right)^{1/2}.
\ee
Note also that 
\be
\|f_j+f_{\infty}\|_{L^2}^2 \le (\delta + 2\|f_\infty\|_{L^2})^2.
\ee
\end{lem}

If $C$ is not monotone in $r$ but one knows it has at most $N$
monotone subsegments then we can sum up the segments 
applying this lemma to each subsegment.

\begin{proof}
Since $C(t)=(r(t), \theta(t))$ is such that $r'(t) > 0$ everywhere then we can reparametrize 
so that $r(t)=r$.  Now by comparing two lengths and taking advantage of the inequality  $\sqrt{|a-b|}\ge |\sqrt{a}-\sqrt{b}|$ we find
\begin{align}
&|L_j(C)-L_\infty(C)| 
\\&\le \int_{r(0)}^{r(1)}\left| \sqrt{ 1 + f_j^2(r)) \theta'(r)^2 }
- \sqrt{ 1 + f_\infty^2(r) \theta'(r)^2 }\right|\, dr \\
&\le  \int_{r(0)}^{r(1)} \sqrt{  |f_j^2(r) -  f_\infty^2(r)|} |\theta'(r)|\, dr \\
&\le \left( \int_{r(0)}^{r(1)}   |f_j^2(r) -  f_\infty^2(r)|\,dr \right)^{1/2}
\left(\int_{r(0)}^{r(1)} |\theta'(r)|^2\, dr \right)^{1/2}
\end{align}
where we used Holder's inequality in the last line.

Now we notice that
\begin{align}
 |f_j^2 -  f_\infty^2| &=  |f_j^2 - f_jf_{\infty}+f_jf_{\infty} -  f_\infty^2| \\
 &= |f_j(f_j - f_{\infty})+f_{\infty}(f_j -  f_\infty)|
 \\&=|(f_j+f_{\infty})(f_j-f_{\infty})| = |f_j+f_{\infty}||f_j-f_{\infty}|.
\end{align}
Combining this with H\"older's Inequality we obtain
\begin{align}
&|L_j(C) -L_\infty(C)| \le \\
& \left( \int_{r(0)}^{r(1)}  |f_j+f_{\infty}|^2\,dr \right)^{1/4}\left( \int_{r(0)}^{r(1)}  |f_j-f_{\infty}|^2\,dr \right)^{1/4} \Theta(C).
\end{align}

Lastly, we notice that
\begin{align}
\|f_j+f_{\infty}\|_{L^2}^2&=\|f_j - f_\infty+2f_{\infty}\|_{L^2} \\
&\le ( \|f_j-f_{\infty}\|_{L^2} + 2\|f_\infty\|_{L^2})^2 \le (\delta + 2\|f_\infty\|_{L^2})^2
\end{align}
which gives us the desired uniform bound.
\end{proof}

Now that we have obtained a bound on fixed geodesics which are monotone in $r$ we would like to gain some control on the term $\Theta(C)$ from Lemma \ref{warp-L} in the case where $C$ is a fixed geodesic with respect to the metric $g_j$. We note that we will use Lemma \ref{Theta(C_j)-WEst} only in the case where $C$ is a fixed geodesic with respect to $g_{\infty}$ which is monotone in $r$ but we state it in more generality below since it could be useful for future results.

\begin{lem}\label{Theta(C_j)-WEst}
Let $M_j$ be a warped product manifold as in \ref{MorN}.
Let $C_j(t) = (r(t),\theta(t))$ be a unit speed absolutely continuous geodesic
 in $M_j$ which is non-decreasing in $r$ and define
\be
\displaystyle m_j = \min_{r \in [r_0,r_1]}f_j(r) >0.
\ee
Then $\Theta$ of (\ref{eq-Theta}) satisfies:
\be
\Theta(C_j) \le \frac{\sqrt{n-1} L_j(C_j)^{1/2}}{m_j}.
\ee
\end{lem}

\begin{proof}
We can estimate $\Theta(C_j)$ by rewriting the line integral which defines $\Theta(C_j)$
\begin{align}
\Theta(C_j) &=\left( \int_{r(0)}^{r(1)} |\vec{\theta}'(r)|^2\, dr \right)^{1/2} = \left (\int_0^{L_j(C_j)}|\vec{\theta}'(t)|^2 r'(t) dt \right )^{1/2} .
\end{align}
Now by the assumption that $|C_j'|_{g_j}=\sqrt{r'(t)^2+f_j(r(t))^2|\vec{\theta}_j'(t)|^2}=1$ a.e. and $r'(t) >0$ we find that $0 <r'(t) \le 1 $ which yields
\begin{align}
\Theta(C_j)&\le\left (\int_0^{L_j(C_j)}|\vec{\theta}'(t)|^2  dt \right )^{1/2} .
\end{align}
 
Note that $|C_j'|_{g_j}=\sqrt{r'(t)^2+f_j(r(t))^2|\vec{\theta}_j'(t)|^2}=1$ a.e. implies that $|\vec{\theta}_j'(t)| \le \frac{1}{ f_j}$ a.e. which yields the estimate
\begin{align}
\Theta(C_j)&\le \left (\int_0^{L_j(C_j)}\frac{1}{f_j(r(t))^2}  dt \right )^{1/2} \le\frac{ L_j(C_j)^{1/2}}{m_j}.
\end{align}

\end{proof}

\begin{cor} \label{MontoneRUniform}
If the length minimizing absolutely continuous geodesic between $p,q \in M$ with respect to $g_{\infty}$ is monotone in $r$ and we let $\delta= \|f_j-f_{\infty}\|_{L^2}$ and $\displaystyle m_{\infty} = \min_{r \in [r_0,r_1]}f_{\infty}(r) >0$ then we find the uniform estimate
\begin{align}
d_{g_j}(p,q)-d_{g_{\infty}}(p,q) \le \left(\delta^2 + 4\|f_\infty\|_{L^2}^2\right) \delta^{1/2} \frac{\sqrt{n} Diam(M_{\infty})}{m_{\infty}}.
\end{align}
\end{cor}
\begin{proof}
We note that by the fact that $C$ is the length minimizing geodesic between $p,q \in M$ with respect to $g_{\infty}$ we find
\begin{align}
d_{g_j}(p,q)-d_{g_{\infty}}(p,q) \le L_j(C) - L_{\infty}(C).
\end{align}
Now if we combine Lemma \ref{lem-diam}, Lemma \ref{warp-L} and Lemma \ref{Theta(C_j)-WEst} then we find
\begin{align}
d_{g_j}(p,q)-d_{g_{\infty}}(p,q) \le \left(\delta^2 + 4\|f_\infty\|_{L^2}^2\right) \delta^{1/2} \frac{\sqrt{n} Diam(M_{\infty})}{m_{\infty}},
\end{align}
where $\delta= \|f_j-f_{\infty}\|_{L^2}$ and $\displaystyle m_{\infty} = \min_{r \in [r_0,r_1]}f_{\infty}(r) >0$.
\end{proof}

The uniform control of Corollary \ref{MontoneRUniform} will be used in the proof of Theorem \ref{WarpConv} below. Now we would like to control the length of geodesics with respect to $g_{\infty}$ which are constant in $r$. 

\begin{lem}\label{ConstRLengthBound}
Let $p,q \in [r_0,r_1]\times \Sigma$ and assume that the absolutely continuous geodesic $C$ between $p$ and $q$ with respect to $g_{\infty}$ is parameterized as $C = (\hat{r}, \theta(t))$, $t \in [0,1]$, for some fixed $\hat{r} \in [r_0,r_1]$. If $f_j\rightarrow f_{\infty}$ in $L^2$ then
\begin{align}
\limsup_{j \rightarrow \infty} d_{g_j}(p,q) \le d_{g_{\infty}}(p,q).
\end{align}
Moreover, we can find an approximating curve $C_j^{\epsilon}$ between $p$ and $q$ so that
\begin{align}
L_j(C_j^{\epsilon}) \le 4 \delta_j^{\epsilon} + L_{\infty}(C) + \epsilon d_{\sigma}(\theta(0),\theta(1)),
\end{align}
where 
\begin{align}
\delta_j^{\epsilon}\le \frac{|f_j-f_{\infty}|_{L^2}^2}{\epsilon^2}.
\end{align}
\end{lem}
\begin{proof}
Since $f_j \rightarrow f_{\infty}$ in $L^2$ if we define 
\begin{align}
S_{\epsilon}^j = \{ x \in [r_0,r_1]: |f_j(x)-f_{\infty}(x)| \ge \epsilon\}
\end{align}
then we know that there exists a $\delta_j> 0$ so that $|S_{\epsilon}^j| \le \delta_j$, where $\delta_j \rightarrow 0$ as $j \rightarrow \infty$. This follows since if $|S_{\epsilon}^j| \ge c > 0$ then 
\begin{align}
\int_{-\pi}^{\pi}|f_j-f_{\infty}|^2 dr \ge \int_{S_{\epsilon}^j}|f_j-f_{\infty}|^2 dr \ge c\epsilon^2
\end{align}
which leads to a contradiction. In fact, 
\begin{align}
\epsilon |S_j^{\epsilon}| &\le \int_{S_j^{\epsilon}}|f_j-f_{\infty}| dr 
\\&\le |S_j^{\epsilon}|^{1/2} \left( \int_{S_j^{\epsilon}}|f_j-f_{\infty}|^2 dr\right)^{1/2} 
\\&\le |S_j^{\epsilon}|^{1/2} \left( \int_{-\pi}^{\pi}|f_j-f_{\infty}|^2 dr\right)^{1/2},
\end{align}
which implies
\begin{align}
\delta_j \le \frac{|f_j-f_{\infty}|_{L^2}^2}{\epsilon^2}.
\end{align}
This implies that we can choose an $r_j \in(\hat{r},\hat{r}+2 \delta_j)$ or $r_j\in(\hat{r}-2\delta_j, \hat{r})$ so that $|f_j(r_j) - f_{\infty}(r_j)| \le \epsilon$ and so by combining with Lemma \ref{distEst-L} and Lemma \ref{AvoidPeak2} we find a curve $C_j^{\epsilon}$ between $p$ and $q$ so that
\begin{align}
d_{g_j}(p,q) &\le L_j(C_j^{\epsilon})
\\&\le 4 \delta_j + f_j(r_j) d_{\sigma}(\theta(0),\theta(1))
\\&\le 4 \delta_j+f_{\infty}(r_j)d_{\sigma}(\theta(0),\theta(1)) 
\\&\qquad + |f_j(r_j) - f_{\infty}(r_j)|d_{\sigma}(\theta(0),\theta(1)).
\end{align}
Now by taking limits as $j \rightarrow \infty$ and using that $f_{\infty}$ is continuous we find
\begin{align}
\limsup_{j \rightarrow \infty}d_{g_j}(p,q) \le  f_{\infty}(\hat{r}) d_{\sigma}(\theta(0),\theta(1))+ \epsilon d_{\sigma}(\theta(0),\theta(1)).
\end{align}
Since this is true for all $\epsilon > 0$ and $d_{g_{\infty}}(p,q) = f_{\infty}(\hat{r}) d_{\sigma}(\theta(0),\theta(1))$ the desired result follows.
\end{proof}

We now combine these lemmas into a proposition:

\begin{prop} \label{prop-limsup}
If $f_j$ and $f_\infty$ are positive continuous functions,
$f_j \rightarrow f_{\infty}$ in $L^2$, and $M_j=M$ are warped products as in \eqref{MorN}  
 then
\begin{align}
\limsup_{j \rightarrow \infty} d_j(p,q)\le  d_{\infty}(p,q) \textrm{pointwise}.
\end{align}
\end{prop}

\begin{proof}

Fix $p$ and $q$ in $M_j=M$.  Let $C(t)$ be a minimizing curve between
 $p$ and $q$ with respect to $g_{\infty}$:
 \be
 L_\infty(C) = d_\infty(p,q).
 \ee
By Remark~\ref{s-limit-geods},
 $C$ is an absolutely continuous curve. It can be broken down into possibly infinitely many segments, each of which 
is either monotone in $r$ or has constant $r$ component. Let $\mathcal{C}=\{C^{\alpha}: \alpha \in I\}$, where $I$ is an indexing set, be the segments which are constant in $r$ with endpoints $(r^{\alpha},\theta_1^{\alpha}),(r^{\alpha},\theta_2^{\alpha})\in [r_0,r_1]\times \Sigma$ then we can estimate
\begin{align}
L_{\infty}(C) &\ge \sum_{\alpha \in I} L_{\infty}(C^{\alpha})
\\&= \sum_{\alpha \in I} f_{\infty}(r^{\alpha}) d_{\sigma}(\theta_1^{\alpha},\theta_2^{\alpha}) \ge \left(\min_{r\in[r_0,r_1]} f_{\infty}(r) \right)\sum_{\alpha \in I}  d_{\sigma}(\theta_1^{\alpha},\theta_2^{\alpha}),
\end{align}
and hence 
\begin{align}\label{CountablyManyTheta}
\sum_{\alpha \in I}  d_{\sigma}(\theta_1^{\alpha},\theta_2^{\alpha}) \le \frac{Diam(M_{\infty})}{\left(\min_{r\in[r_0,r_1]} f_{\infty}(r) \right)}< \infty.
\end{align}
Similarly, if we let $\tilde{\mathcal{C}}=\{\tilde{C}^{\alpha}: \alpha \in I\}$ be the collection of segments of $C$ which are monotone in $r$, with endpoints $(r_1^{\alpha},\theta_1^{\alpha}),(r_2^{\alpha},\theta_2^{\alpha})\in [r_0,r_1]\times \Sigma$, then
\begin{align}
L_{\infty}(C) &\ge \sum_{\alpha \in I} L_{\infty}(\tilde{C}^{\alpha}) 
\\&= \sum_{\alpha \in I} \int_{r_1^{\alpha}}^{r_2^{\alpha}} \sqrt{1 + f_{\infty}(r)^2 \theta'(r)^2} dr 
\\&\ge \sum_{\alpha \in I}\int_{r_1^{\alpha}}^{r_2^{\alpha}} dr = \sum_{\alpha \in I}|r_1^{\alpha}-r_2^{\alpha}|,
\end{align}
which implies
\begin{align}\label{CountableManyR}
\sum_{\alpha \in I}|r_1^{\alpha}-r_2^{\alpha}| \le Diam(M_{\infty}).
\end{align}
So, by combining \eqref{CountablyManyTheta}, \eqref{CountableManyR}, and Lemma \ref{distEst-L} we find for any $\eta > 0$, we can choose $I_{\eta} \subset I$, $I \setminus I_{\eta} = K \in \N$, so that
\begin{align}\label{WorstCaseBound}
 &\sum_{\alpha \in I_{\eta}} L_{\infty}(\tilde{C}^{\alpha}) + \sum_{\alpha \in I_{\eta}} L_{\infty}(C^{\alpha})   
 \\& \qquad\le \sum_{\alpha \in I_{\eta}}|r_1^{\alpha}-r_2^{\alpha}| + 2\left(\max_{r\in[r_0,r_1]} f_{\infty}(r) \right)\sum_{\alpha \in I_{\eta}}  d_{\sigma}(\theta_1^{\alpha},\theta_2^{\alpha}) \le \eta
\end{align}
and hence by replacing all but finitely many subsegments of $C$ with finitely many taxi minimizing curves whose $g_{\infty}$ length is smaller than $\eta$ we can obtain another curve $\bar{C}^{\eta}$ so that 
\begin{align}
L_{\infty}(\bar{C}^{\eta}) \le L_{\infty}(C) - 2\eta.
\end{align}
This can be done so that $\bar{C}^{\eta}$ can be broken down into finitely many segments, each of which is either monotone in $r$ or has constant $r$ component. By Lemma \ref{ConstRLengthBound}, for each monotone segment $\bar{C}^k$, $k \in \N$, $k \le K$ we can find an approximating curve, $\bar{C}_j^{k,\epsilon}$, so that
\begin{align}\label{MonotonSegAdjustment}
L_j(\bar{C}_j^{k,\epsilon}) \le 4 \delta_j^{\epsilon} + L_{\infty}(\bar{C}^k) + \epsilon d_{\sigma}(\theta_1^k,\theta_2^k),
\end{align}
where $\delta_j^{\epsilon} \le \frac{|f_j-f_{\infty}|_{L^2}^2}{\epsilon^2}$.

Then by Lemmas \ref{warp-L}, \ref{Theta(C_j)-WEst} and \ref{ConstRLengthBound} we can find a curve $\bar{C}_j^{\eta,\epsilon}$, $\epsilon > 0$ between $p$ and $q$, by possibly adjusting the monotone segments as in \eqref{MonotonSegAdjustment}, so that
\be \label{EtaStillAllive}
\limsup_{j\to \infty} L_j(\bar{C}_j^{\eta,\epsilon}) \le L_\infty(C)-2\eta + \epsilon \frac{Diam(M_{\infty})}{\left(\min_{r\in[r_0,r_1]} f_{\infty}(r) \right)}.
\ee

Since \eqref{EtaStillAllive} is true for all $\eta$, $d_j(p,q) \le L_j(\bar{C}_j^{\eta,\epsilon})$ and $L_\infty(C)=d_{\infty}(p,q)$ we have
\begin{align}
\limsup_{j \rightarrow \infty} d_j(p,q)\le  d_{\infty}(p,q)+ \epsilon \frac{Diam(M_{\infty})}{\left(\min_{r\in[r_0,r_1]} f_{\infty}(r) \right)},
\end{align}
which is true for all $\epsilon >0$ and hence the desired result follows.
\end{proof}

\subsection{Proof of Theorem~\ref{WarpConv}}

Recall that in the statement of Theorem~\ref{WarpConv} we have a sequence of
warping functions $f_j(r) \ge f_{\infty}(r)-\frac{1}{j}$
and $f_j(r) \rightarrow f_{\infty}(r)$ in $L^2$. 
We will prove:
 \begin{align}
\lim_{j \rightarrow \infty} d_j(p,q) = d_{\infty}(p,q)
\end{align}
uniformly by first showing it converges pointwise on a subsequence and then
applying Theorem~\ref{HLS-thm} which implies uniform
convergence, GH and $\mathcal{F}$ convergence
to the same space.

 \begin{prop} \label{prop-ptwise}
 Under the hypothesis of Theorem~\ref{WarpConv} we have pointwise
 convergence of the distance functions:
 \begin{align}
\lim_{j \rightarrow \infty} d_j(p,q) = d_{\infty}(p,q)\label{pointwiseDistConv}
\end{align}
\end{prop}
 
 \begin{proof}
Let $p,q \in [r_0,r_1]\times \Sigma$.  Applying the $C^0$ lower bound and Lemma \ref{distLowerBound} we have
\begin{align}\label{distliminf}
\liminf_{j \rightarrow \infty} d_{j}(p,q) \ge d_{\infty}(p,q)
\end{align}
Applying the $L^2$ upper bound and Proposition~\ref{prop-limsup} we also have
\begin{align}
\limsup_{j \rightarrow \infty} d_j(p,q)\le  d_{\infty}(p,q).
\end{align}
Thus we have pointwise convergence.
\end{proof}

We now prove Theorem~\ref{WarpConv}:

\begin{proof}
By the assumption that $0< c \le f_{\infty} - \frac{1}{j} \le f_j \le K$ we can use Lemma \ref{biLip} and choose $\lambda = \max\left(\frac{1}{\min(c,1)},\max(1,K) \right) > 0$ so that for $j$ large enough we find
\begin{align}
\lambda \ge \frac{d_j(p,q)}{d_1(p,q)}\ge \frac{1}{\lambda},\label{distLipBound3}
\end{align}
where $d_1$ is the distance defined with warping factor $1$.

Now can apply Theorem \ref{HLS-thm} to conclude that there exists a length metric $d'_\infty$ and a subsequence $d_{j_k}$ so that $d_{j_k}$ converges uniformly to $d'_{\infty}$, and hence GH and SWIF converges as well. By the pointwise convergence proven in Proposition~\ref{prop-ptwise}, we know that $d'_{\infty} = d_{\infty}$ and hence $d_{j_k}$ must uniformly converge to $d_{\infty}$. Since this is true for all the subsequences, we see that $d_j$ uniformly converges to $d_{\infty}$.  Appealing again to Theorem~\ref{HLS-thm} we see it converges in the Gromov-Hausdorff and intrinsic flat sense as well. 
\end{proof}

\section{Warping functions with two variables on Tori}\label{sect-WarProd2Var}
 
 In this section we give a short exploration of more general warped product manifolds.  
 There are a wealth of new directions one might explore and this section demonstrates how some of our techniques do extend easily.  Here we prove the following theorem:
 
 \begin{thm}\label{WarpConv2}
 Let $g_j = dx^2+dy^2+f_j(x,y)^2dz^2$ be a  metric on a torus $M_j = {\mathbb{S}}^1 \times {\mathbb{S}}^1 \times_{f_j}{\mathbb{S}}^1$ with coordinates $(x,y,z) \in [-\pi,\pi]^3$, $f_j \in C^0([-\pi,\pi]^2)$. Assume that, 
 \begin{align}
  &f_j \rightarrow f_{\infty}= c > 0 \text{ in } L^2, 
  \\0< &f_{\infty} - \frac{1}{j} \le f_j \le K < \infty,
  \end{align}
   then $M_j$ converges uniformly to $M_{\infty}$ as well as 
   \begin{align}
   M_j &\GHto M_{\infty},
   \\M_j &\Fto M_{\infty}.
   \end{align}
 \end{thm}
 
This theorem will be applied in upcoming joint work of a team of doctoral students who are working with the first author: Lisandra Hernandez-Vazquez, Davide Parise, Alec Payne, and Shengwen Wang.  Various members of this team which first began working together at the Fields Institute in the Summer of 2017 will explore further theorems in this direction using similar techniques.

The proof of this theorem will follow similar to the proof of Theorem~\ref{WarpConv}
however we have some additional difficulties arising.  The main difficulty is that $f_j \rightarrow f_{\infty}$ in  $L^2([-\pi,\pi]^2)$ does not imply that $f_j \rightarrow f_{\infty}$ on curves and hence we will not be able to prove the corresponding results to Lemma \ref{warp-L} and \ref{Theta(C_j)-WEst} for this setting. Instead in Lemmas \ref{ConstzControl}, \ref{ConstxyControl}, and \ref{MonotoneControl} we will build approximating sequences of curves to a geodesic with respect to $g_{\infty}$ and show $\displaystyle \limsup_{j\rightarrow \infty} d_j(p,q) \le d_{\infty}(p,q)$. The $C^0$ control on $f_j$ works similarly to section \ref{sect-WarProd} and hence we are able to show $\displaystyle \liminf_{j\rightarrow \infty} d_j(p,q) \ge d_{\infty}(p,q)$ in Lemma \ref{distLowerBound2} This will imply pointwise convergence of distances which when combined with Theorem \ref{HLS-thm} will show uniform, GH ans SWIF convergence, similar to the examples in section \ref{Sect_Ex}.

\subsection{A lower $C^0$ bound}

We now prove a lemma which shows the consequence of a $C^0$ lower bound which we have seen is important by the examples in section \ref{Sect_Ex}.
 
 \begin{lem}\label{distLowerBound2} Let $p,q \in M_j$ and assume that 
 \be
 f_{j}(x,y) \ge f_{\infty}(x,y) - \frac{1}{j} > 0 \textrm{ and } diam(M_j) \le D.
 \ee 
 Then
\begin{align}
\liminf_{j \rightarrow \infty} d_{j}(p,q) \ge d_{\infty}(p,q)
\end{align}
\end{lem}
 \begin{proof} 
Let $C_j(t)=(x_j(t),y_j(t),z_j(t))$ be the minimizing absolutely continuous geodesic in $M_j$, parameterized so that $|C_j'(t)|_{g_j}=1$ a.e., realizing the distance between $p$ and $q$ then compute
\begin{align}
 g_j(C_j'(t),C_j'(t)) &= {x_j'(t)^2+y_j'(t)^2+f_j(x_j(t),y_j(t))^2 |z_j'(t)|^2}
 \\&\ge  {x_j'(t)^2+y_j'(t)^2+(f_{\infty}(x_j(t),y_j(t)) - \tfrac{1}{j})^2 |z_j'(t)|^2}
  \\&= x_j'(t)^2+y_j'(t)^2+f_{\infty}(x_j(t),y_j(t))^2|z_j'(t)|^2 
  \\& \quad -  \left(\tfrac{2}{j}f_{\infty}(x_j(t),y_j(t))|z_j'(t)|^2 - \tfrac{1}{j^2} |z_j'(t)|^2\right ) 
  \end{align}
  Note that the terms here are positive by the assumptions of the lemma, so that when we take the
  square root we can apply
  the inequality 
  \be
  \sqrt{|a-b|}\ge |\sqrt{a}-\sqrt{b}|\ge \sqrt{a}-\sqrt{b},
  \ee 
  before integrating to obtain
  \begin{align}
  \qquad d_{g_j}&(p,q) 
 =   \int_0^{L_j(C_j)} \sqrt{g_j(C_j'(t),C_j'(t)) }\, dt 
   \\&\ge \int_0^{L_j(C_j)} 
         \sqrt{x_j'(t)^2+y_j'(t)^2+f_{\infty}(x_j(t),y_j(t))^2|z_j'(t)|^2}dt 
  \\& - \int_0^{L_j(C_j)}
     \sqrt{\tfrac{2}{j}f_{\infty}(x_j(t),y_j(t))|z_j'(t)|^2 - \tfrac{1}{j^2} |z_j'(t)|^2 } \,dt
    \\&\ge L_{g_{\infty}}(C_j)- \tfrac{1}{\sqrt{j}}\int_0^{L_j(C_j)}|z_j'(t)|\sqrt{\left(2f_{\infty}(x_j(t),y_j(t)) - \tfrac{1}{j} \right )}\,dt
 \end{align} 
Now we notice that 
\begin{align}
|C_j'(t)|_{g_j} &=\sqrt{x_j'(t)^2+y_j'(t)^2+f_j(x_j(t),y_j(t))^2|z_j'(t)|^2} =1 \text{ a.e.}
\\\Rightarrow |z_j'(t)| &\le \frac{1}{f_j(x_j(t),y_j(t))} \text{ a.e.}
\end{align}
and hence we can then conclude that
 \be
  d_{g_j}(p,q) \ge d_{g_{\infty}}(p,q) -\frac{\sqrt{2}\max_{[-\pi,\pi]^2}\sqrt{f_{\infty}} D}{\min_{[-\pi,\pi]^2}f_{j}\sqrt{j}}.
\ee
 The desired result follows by taking limits.
\end{proof}

We now prove that we have uniform bounds on the diameter which was used in Lemma \ref{distLowerBound2}:

 \begin{lem} \label{lem-diam2} 
 If $\|f_j-f_{\infty}\|_{L^2} \le \delta_j$ and $M_j$ are warped products as in Theorem \ref{WarpConv2}  
 then
 \be
  \diam(M_j) \le 4\sqrt{2}\pi 
  + 2 \pi\left( \|f_\infty\|_{C_0} + \tfrac{\delta_j}{2 \pi }\right) .
  \ee
  \end{lem}
  
  \begin{proof}
  Let $p,q\in M_j$ with $p = (x_1,y_1,z_1)$ and $q=(x_2,y_2,z_2)$.   Recall that the distance between these
  points is the infimum over lengths of all curves.  For any $(x_0,y_0) \in [-\pi,\pi]^2$ we can take a 
  first path from $p$ to $(x_0,y_0,z_1)$ which stays in a plane parallel to the $xy-$plane, then a second path from $(x_0,y_0,z_1)$ to $(x_0,y_0,z_2)$ parallel to the $z$ axis,
  and then a third path from $(x_0,y_0,z_2)$ to $(x_2,y_2,z_2)$ which stays in a plane parallel to the $xy-$plane.  The first and third paths each have
   length $\le 2\sqrt{2}\pi$, and the middle path has length bounded above by $2 \pi$ with respect to the flat metric.  Thus we have
  \begin{eqnarray}
  \qquad d_j(p,q) &\le& 4\sqrt{2}\pi + 2 \pi f_j(x_0,y_0)  \\
  &\le & 4\sqrt{2}\pi + 2 \pi \left(\, f_\infty(x_0,y_0) + |f_j(x_0,y_0)-f_\infty(x_0,y_0)|\, \right) \,.
  \end{eqnarray}
  Choosing an $(x_0,y_0)$ such that 
  \be
  |f_j(x_0,y_0)-f_\infty(x_0,y_0)|^2 \le \frac{1}{4 \pi^2} \int_{-\pi}^{\pi} \int_{-\pi}^{\pi}|f_j(x,y)-f_\infty(x,y)|^2 \, dxdy
  \ee
  we have
  \be
  |f_j(x_0,y_0)-f_\infty(x_0,y_0)| \le \frac{\|f_j-f_\infty\|_{L_2}}{2 \pi}
  \ee
  and  $f_\infty(x_0,y_0) \le \|f_\infty\|_{C_0}$.
  \end{proof}

\subsection{$L^2$ convergence and convergence of distances}

In this section we will build sequences of curves whose length approximates the length of a fixed geodesic with respect to $g_{\infty}$ whose warping function is a constant. 

We start by approximating a geodesic which has constant $z$ component which is simple since $g_j$ agrees with $g_{\infty}$ in the $x$ and $y$ directions.
 
  \begin{lem}\label{ConstzControl}
Let $p, q \in [-\pi,\pi]^3$ so that $p=(x_1,y_1, z_0)$ and $q=(x_2,y_2,z_0)$. If $f_{\infty} = c > 0$ then we have that
\begin{align} 
\limsup_{j \rightarrow \infty} d_{j}(p,q) \le d_{\infty}(p,q).
\end{align}
 \end{lem}
 
 \begin{proof}
 Let $\gamma$ be a minimal geodesic with respect to $g_{\infty}$ from $p$ to $q$. Since $g_{\infty}$ is a Euclidean metric it is a straight line segment:
 \begin{align}
 \gamma(t)=(x_1(1-t)+x_2t, y_1(1-t)+y_2t,z_0),
 \end{align}
 Note that we can choose coordinate so that this is the minimal geodesic with respect to $g_{\infty}$. Then we can compute,
 \begin{align}
 d_j(p,q) &\le L_j(\gamma)= \int_0^1 \sqrt{(x_2-x_1)^2+(y_2-y_1)^2}dt= d_{\infty}(p,q) ,
 \end{align}
 since $g_j$ agrees with $g_{\infty}$ in the $x$ and $y$ directions, by which the result follows by taking limits.
 \end{proof}
 
 We now construct a sequence of curves which approximates a fixed geodesic with respect to $g_{\infty}$ which is constant in $x$ and $y$.
 
 \begin{lem}\label{ConstxyControl}
Assume that $f_j \rightarrow f_\infty = c > 0$ in $L^2$ and let  $p, q \in [-\pi,\pi]^3$ so that $p=(x_0,y_0, z_1)$ and $q=(x_0,y_0,z_2)$ then we have that
\begin{align} 
\limsup_{j \rightarrow \infty} d_{j}(p,q) \le d_{\infty}(p,q).
\end{align}
 \end{lem}
 
 \begin{proof}
 We claim that if 
 \begin{align}
 S_{\epsilon}^j = \{ (x,y) \in [-\pi,\pi]^2: |f_j(x,y)-f_{\infty}(x,y)|\ge \epsilon\}
 \end{align}
 then we must have that $|S_{\epsilon}^j| \le \delta_j$ where $\delta_j \rightarrow 0$ as $j \rightarrow \infty$ ($|S|$ represents Lebesgue measure of $S \subset [-\pi,\pi]^2$ with respect to the Euclidean metric). 
 If the claim were false then $|S_{\epsilon}^j| \ge C > 0$ and
 \begin{align}
\int_{-\pi}^{\pi}\int_{-\pi}^{\pi} |f_j(x,y)-f_{\infty}(x,y)|^2dx dy \ge \int_{S_{\epsilon}^j} |f_j(x,y)-f_{\infty}(x,y)|^2dA \ge C \epsilon^2
 \end{align}
 which contradicts $f_j \rightarrow f_{\infty}$ in $L^2$.
 
 Define the set
 \begin{align}
 T_{\epsilon}^j = \left(B\left((x_0,y_0),4\sqrt{\delta_j}\right) \setminus S_{\epsilon}^j\right)\cap [-\pi,\pi]^2.
 \end{align}
 Since eventually 
 \begin{align}
 \frac{|B((x_0,y_0),4\sqrt{\delta_j)}|}{4} =4 \pi  \delta_j > |S_{\epsilon}^j|,
 \end{align}
  we see that
 $T_{\epsilon}^j$ is non-empty .  Hence we can choose a $(x_{\epsilon}^j,y_{\epsilon}^j) \in T_{\epsilon}^j$.  
 
 \begin{figure} [h]
\centering
\includegraphics[width=4in]{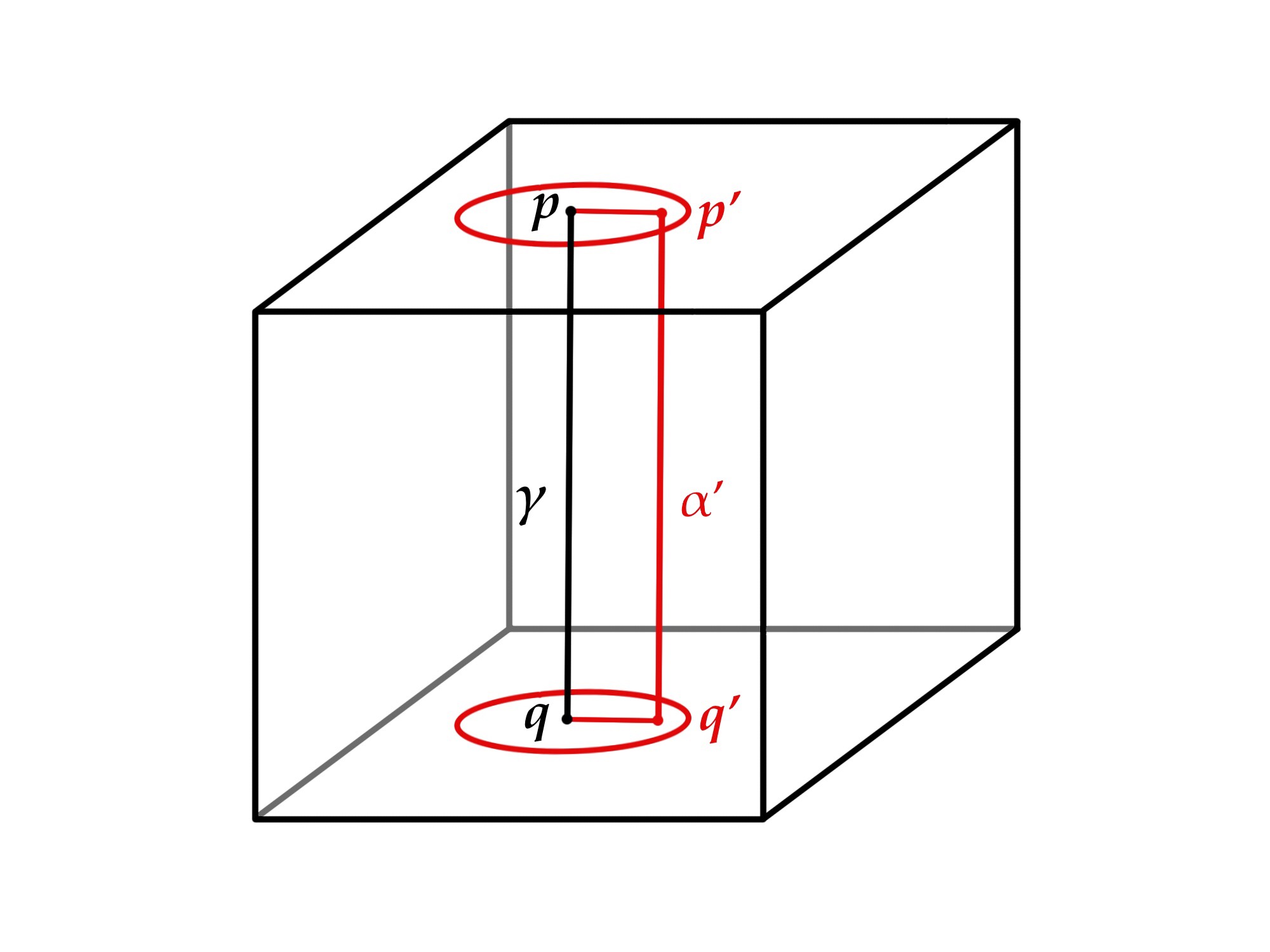}
\caption{$\alpha' = \alpha^j_{x_j}$ approximates the curve $\gamma$ between the points $p$ and $q$.}
\label{fig-cube-vert}
\end{figure}

A minimal geodesic $\gamma$ from $p= (x_0,y_0,z_1)$ to $q = (x_0,y_0,z_2)$
with respect to $g_\infty$ is purely vertical:
\be
\gamma(t)=(x_0, y_0, z_0(1-t)+z_2t)
\ee
where the addition is mod $2\pi$.  Note that 
$d_{\infty}(p,q) =c|z_2-z_1|$.
Let
\be
p'=(x_{\epsilon}^j,y_{\epsilon}^j,z_1) \textrm{ and }q'=(x_{\epsilon}^j,y_{\epsilon}^j,z_2).
\ee
So $d_\infty(p,p') < 4\sqrt{\delta_j}$ and $d_\infty(q,q') < 4\sqrt{\delta_j}$.
Also 
\be
d_{\infty}(p,q) =c|z_2-z_1|= d_\infty(p',q').
\ee

  We can define a curve $\alpha_{\epsilon}^j$ as in Figure~\ref{fig-cube-vert} which
  approximates $\gamma$.  This curve
 runs minimally with respect to $g_\infty$ from $p$ to $p'$ and then minimally to
  $q'$ and then minimally to $q$ as follows:
 \begin{align}
 \alpha_{\epsilon}^j(t) = 
 \begin{cases}
 (x_0(1-3t) + 3x_{\epsilon}^j t,\,y_0(1-3t)+3y_{\epsilon}^jt,\,z_1) & 0\le t \le 1/3
 \\(x_{\epsilon}^j,\, y_{\epsilon}^j,\, z_1(2-3t)+z_2(3t-1)) & 1/3\le t \le 2/3
 \\ (x_{\epsilon}^j(3-3t)+x_0(3t-2),\, y_{\epsilon}^j(3-3t))+ y_0 (3t-2),\, z_2)  & 2/3\le t \le 1
 \end{cases}
 \end{align}
 where the addition here is mod $2\pi$.   
  
Now we can compute
\begin{align}
d_j(p,q) &\le L_j(\alpha_{\epsilon}^j) 
\\&=\int_0^{1/3} \sqrt{|3x_{\epsilon}^j-3x_0|^2 + |3y_{\epsilon}^j-3y_0|^2}dt 
\\&+ \int_{1/3}^{2/3}|3z_2-3z_1|f_j(x_{\epsilon}^j,y_{\epsilon}^j) dt 
\\&+ \int_{2/3}^1 \sqrt{|3x_{\epsilon}^j-3x_0|^2 + |3y_{\epsilon}^j-3y_0|^2}dt.
\end{align}
Combining this with the definitions of $(x_{\epsilon}^j,y_{\epsilon}^j) \in T_{\epsilon}^j$ and using the continuity of $f_{\infty}$ we find
\begin{align}
d_j(p,q)&= 2\sqrt{|x_0-x_{\epsilon}^j|^2 + |y_0 - y_{\epsilon}^j|^2} + f_j(x_{\epsilon}^j,y_{\epsilon}^j) |z_2-z_1|
\\&\le 16 \sqrt{\delta_j} + |f_j(x_{\epsilon}^j,y_{\epsilon}^j) - f_{\infty}(x_{\epsilon}^j,y_{\epsilon}^j)| |z_2-z_1|+f_{\infty}(x_{\epsilon}^j,y_{\epsilon}^j)|z_2-z_1|
\\&\le 16 \sqrt{\delta_j}  + \epsilon |z_2-z_1|+c|z_2-z_1|.
\end{align}
where we are using the hypothesis that $f_{\infty} = c > 0$.

Now by noticing that $d_{\infty}(p,q) =c|z_2-z_1|$ and taking the limit as $j \rightarrow \infty$ we find
\begin{align}
\limsup_{j \rightarrow \infty} d_j(p,q)&\le \epsilon|z_1-z_0|+c|z_1-z_0| =\epsilon|z_1-z_0|+d_{\infty}(p,q)
\end{align}
and since this is true for all $\epsilon > 0$ the result follows.
 \end{proof}

 We now construct a sequence of curves which approximates a fixed geodesic with respect to $g_{\infty}$ which does not fall under the hypotheses of Lemma \ref{ConstzControl} or \ref{ConstxyControl}.
 
  \begin{lem}\label{MonotoneControl}
Assume that $f_j \rightarrow f_\infty = c > 0$ in $L^2$ and let  $p, q \in [-\pi,\pi]^3$ so that $p=(x_1,y_1, z_1)$, $q=(x_2,y_2,z_2)$ and $(x_1,y_1) \not = (x_2,y_2)$ then 
\begin{align} 
\limsup_{j \rightarrow \infty} d_{j}(p,q) \le d_{\infty}(p,q).
\end{align}
 \end{lem}
 
 \begin{proof}
 Without loss of generality we may assume that $y_1 \not = y_2$.
 Let $\gamma$ be the geodesic with respect to $g_{\infty}$ which runs from
 $p$ to $q$.  Since $g_{\infty}$ is a Euclidean metric, we can choose coordinates
 on ${\mathbb{S}}^1 \times {\mathbb{S}}^1 \times {\mathbb{S}}^1$ such that
  \begin{align}
 \gamma(t)=(\alpha(t) ,z_1(1-t)+z_2t),
 \end{align}
 where the addition is mod $2\pi$ and
 \begin{align}
 \alpha(t) =(x_1(1-t)+x_2t, y_1(1-t)+y_2t) \subset [-\pi,\pi]^2.
 \end{align}
 Since $g_\infty=dx^2+dy^2+c^2dz^2$, we have
 \be
 d_\infty(p,q)= \sqrt{(x_2-x_1)^2 + (y_2-y_1)^2}.
 \ee
 
 We construct a family of geodesics parallel to this geodesic 
 running from $p' = (x_1',y_1,z_1)$ to $q' = (x_2+x_1'-x_1,y_2,z_2)$
 where $x_1' \in B(x_1,1) \subset [-\pi,\pi]$ as follows
 \be
  \gamma_{x_1'}(t)=(\alpha_{x_1}'(t),z_1(1-t)+z_2t)
 \ee
  where
 \be
 \alpha_{x_1'}(t)=(x_1'(1-t) + (x_1' + x_2-x_1)t, y_1(1-t)+y_2t)
\ee
where the addition is mod $2\pi$ with values in $[-\pi, \pi)$.
Observe that $\alpha: (x',t) \to (x,y)$ defined by $\alpha(x',t)=\alpha_{x'}(t)$ is
\be
\alpha(x',t)=(x'+ (x_2-x_1) t, y_1 +(y_2-y_1)t)
\ee
so
\be \label{ch-of-var}
dx\wedge dy = ( 1 dx' +(x_2-x_1) dt) \wedge (0dx'+ (y_2-y_1) dt)= (y_2-y_1) dx'\wedge dt.
\ee 
 
Since $f_j \rightarrow f_{\infty}$ in $L^2$ we define
\begin{align}\label{line-g}
\bar{f}_j(x') = \int_{\alpha_{x'}} |f_j-f_{\infty}|^2 \,dt.
\end{align} 
We define the set
 \begin{align}
 S_{\epsilon}^j = \{ x' \in [-\pi,\pi): \bar{f}_j(x')\ge \epsilon\} \subset [-\pi,\pi),
 \end{align}
 and the set
 \begin{align}
 W = \{\alpha_{x'}(t): x' \in [-\pi,\pi) \text{ and } t \in [0,1]\}. 
 \end{align}
 By the definition of the line segments, $\alpha_{x_1'}$,  we have $W \subset (-\pi,\pi]^2$. 
   
Note that the set 
\begin{align}
T_{\epsilon}^j =\left(B(x_1,4\delta_j)\setminus S_{\epsilon}^j\right)\subset  [-\pi,\pi]
\end{align}
is non empty where $\delta_j=|S_{\epsilon}^j|$.   We claim $\delta_j \to 0$ as $j\to \infty$. 
Indeed
we have 
\begin{eqnarray}
\epsilon |S_{\epsilon}^j| &\le & \int_{x'\in S_{\epsilon}^j} \bar{f}(x')  dx' \\
 &\le & \int_{x'=-\pi}^\pi \bar{f}_j(x') dx' \\
&= & \int_{x'=-\pi}^\pi  \int_{\alpha_{x'}} |f_j-f_{\infty}|^2 \,dtdx'. 
\end{eqnarray}
Applying a change of variables as in (\ref{ch-of-var}), 
we have
\begin{eqnarray}
\delta_j &=& (\epsilon)^{-1}\int \int_{W}|f_j - f_{\infty}|^2  |y_2-y_1|^{-1} dy dx' \\
&\le & (\epsilon)^{-1}|y_2-y_1|^{-1} \int_{-\pi}^{\pi}\int_{-\pi}^{\pi}|f_j - f_{\infty}|^2 dy dx,
\end{eqnarray}
which converges to $0$ by the hypothesis that $f_j \to f_\infty$ in $L^2$.

\begin{figure} [h]
\centering
\includegraphics[width=4in]{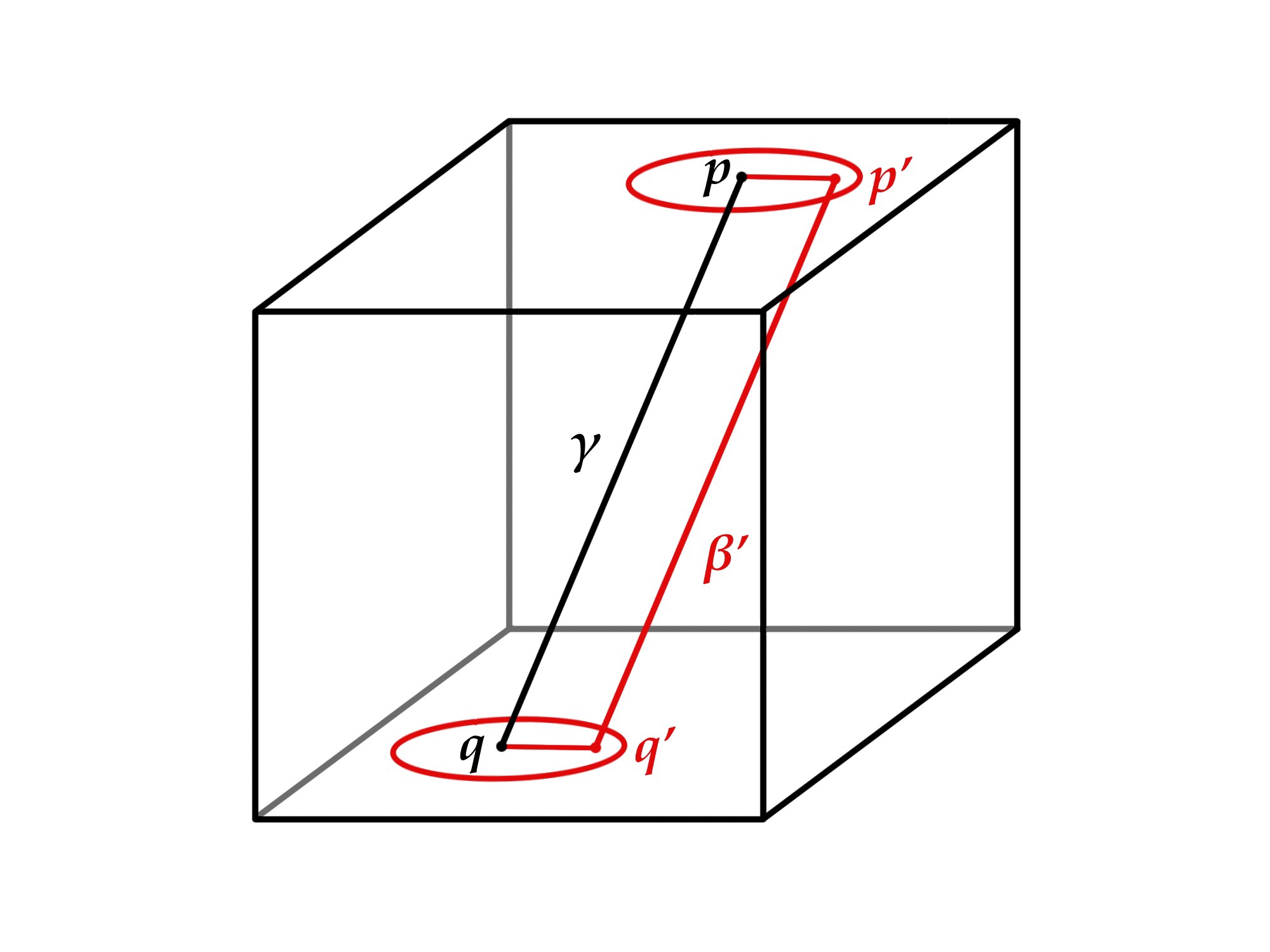}
\caption{$\beta' = \beta^j_{x_j}$ approximates the curve $\gamma$ between the points $p$ and $q$.}
\label{fig-CubeApprox}
\end{figure}

Since $T_{\epsilon}^j$ is nonempty, we can pick a  $x_j \in T_{\epsilon}^j$.  
We use this point to choose
\be
p'=p'_j=(x_{\epsilon}^j,y_{\epsilon}^j,z_1)
\textrm{ and } q'=q_j'=(x_{\epsilon}^j,y_{\epsilon}^j,z_2).
\ee
We can define a sequence of curves $\beta_{x_j}^j$
as in Figure \ref{fig-CubeApprox} which
which
 runs minimally with respect to $g_\infty$ from $p$ to 
 $p'$ and then minimally to
  $q'$ and then minimally to 
  $q$ as follows:
 \begin{align}
 \beta_{x_j}^j(t) = 
 \begin{cases}
 (x_1(1-3t) + 3x_j t,y_1,z_1) & 0\le t \le 1/3
 \\\gamma_{x_j}(3t-1) & 1/3\le t \le 2/3
 \\ ((x_j+x_2-x_1)(3-3t)+x_2(3t-2),y_2,z_2)  & 2/3\le t \le 1.
 \end{cases}
 \end{align}
The sequence of curves $\beta_{x_j}^j(t)$ is the approximating sequence to $\gamma$ which can be used to estimate $d_j(p,q)$ as follows
\begin{eqnarray*}
d_j(p,q) \le L_j(\beta_{x_j}) &=&\int_0^{1/3} |3x_j-3x_1|dt' 
\\&&+ \int_{1/3}^{2/3}\sqrt{|3\Delta x|^2 +|3 \Delta y|^2+ |3 \Delta z|^2 f_j^2(\alpha_{x_j}(3t'-1))} dt' 
\\&&+ \int_{2/3}^1 \sqrt{|3x_2-3(x_j+ x_2-x_1)|^2 }\,dt'
\end{eqnarray*}
where $\Delta x = |x_2-x_1|$,   $\Delta y=|y_2-y_1|$, and $\Delta z =|z_2-z_1|$.   Integrating
the first and last term, and taking $t=3t'-1$ we have
\begin{eqnarray*}
d_j(p,q) &\le & (1/3-0) |3x_j-3x_1| + (1-2/3)\sqrt{|3x_2-3x_j-3x_2+3x_1)|^2 }\\
\\&&+ \int_{0}^{1}\sqrt{|\Delta x|^2 +|\Delta y|^2+ |\Delta z|^2 f_j^2(\alpha_{x_j}(t)} \,dt\\
&\le& |x_j-x_1|+|x_j-x_1|+ \int_0^1\sqrt{\Delta x^2 +\Delta y^2+ \Delta z^2f_j^2(\alpha_{x_j}(t'))}  \,dt 
\\&&\le 2|x_j-x_1| + \int_0^1\sqrt{\Delta x^2 +\Delta y^2+ \Delta z^2f_{\infty}^2 + \Delta z^2 (f_j^2(\alpha_{x_j}(t)) - f_{\infty}^2) }  \,dt\\
&&\le 4 \delta_j+ \int_0^1\sqrt{\Delta x^2 +\Delta y^2+ \Delta z^2f_{\infty}^2}
\, dt  + \int_0^1\Delta z \sqrt{f_j^2(\alpha_{x_j}(t)) - f_{\infty}^2 } \, dt.
\end{eqnarray*}
Since $g_{\infty}$ is Euclidean, the middle term is $d_{\infty}(p,q)$.
Applying H\"older's inequality to the last term of yields
\be\label{last-dj}
d_j(p,q)\le 4\delta_j + d_{\infty}(p,q) +\Delta z \left (\int_0^1|f_j^2(\alpha_{x_j}(t)) - f_{\infty}^2|dt \right) ^{1/2}.
\ee
Recall that we chose $x_j\in T_\epsilon^j$ near $x$ so that 
$x_j \notin S_\epsilon^j$. Thus \eqref{line-g} implies that 
 \begin{align}
\int_{\alpha_{x_j}}|f_j - f_{\infty}|^2dt 
\int_0^1|f_j(\alpha_{x_j}(t)) - f_{\infty}|^2dt= \bar{f}_j(x_j) < \epsilon. 
\end{align}
 We can apply this to control the final term in (\ref{last-dj}) by
factoring and the applying H\"older's inequality and the triangle inequality 
\begin{eqnarray*}
\left(\int_{\alpha_{x_j}}|f_j^2 - f_{\infty}^2|dt\right)^{1/2} 
&\le &
\left(\int_{\alpha_{x_j}}|f_j- f_{\infty}||f_j+ f_{\infty}|dt \right)^{1/2}\\
&\le &  \left (\int_{\alpha_{x_j}}|f_j - f_{\infty}|^2dt \right) ^{1/4} 
\left (\int_{\alpha_{x_j}}|f_j+ f_{\infty}|^2 dt\right) ^{1/4}\\
&\le & \epsilon ^{1/4} \left (\int_{\alpha_{x_j}}|f_j - f_{\infty}+ 2f_\infty|^2dt \right) ^{1/4} 
\\
&\le &  \epsilon^{1/4} \left (\int_{\alpha_{x_j}}\left(|f_j - f_{\infty}|+ 2|f_\infty|\right)^2dt \right) ^{1/4} \\
&= &  \epsilon^{1/4} \left (\int_{\alpha_{x_j}}|f_j - f_{\infty}|^2
+ 4 |f_j - f_{\infty}| \, |f_\infty| + 4|f_\infty|^2 dt \right) ^{1/4} \\
&\le & \epsilon^{1/4} \left ( \epsilon
+ 4 c \int_{\alpha_{x_j}} |f_j - f_{\infty}|  \,dt  + 4c^2 \right) ^{1/4} \\
&\le & \epsilon^{1/4} \left ( \epsilon
+ 4c \left( \int_{\alpha_{x_j}} |f_j - f_{\infty}|^2 \, dt\right)^{1/2}  + 4c^2 \right) ^{1/4} \\
&\le & \epsilon^{1/4} \left ( \epsilon
+ 4\,c\, \epsilon^{1/2}  + 4c^2 \right) ^{1/4}.
\end{eqnarray*}
Substituting this into (\ref{last-dj}) we have
 \be
d_j(p,q)\le 4\delta_j + d_{\infty}(p,q) +\Delta z  \epsilon^{1/4} \left ( \epsilon
+ 4\,c\, \epsilon^{1/2}  + 4c^2 \right) ^{1/4}.
\ee
Now by taking limits as $j \rightarrow \infty$ we find
\begin{align}
\limsup_{j\rightarrow \infty} d_j(p,q)&\le d_{\infty}(p,q) +\Delta z \epsilon^{1/4} \left ( \epsilon
+ 4\,c\, \epsilon^{1/2}  + 4c^2 \right) ^{1/4}.
\end{align}
Since this is true for all $\epsilon > 0$ the lemma follows.
 \end{proof}
 
\subsection{Proof of Theorem \ref{WarpConv2}.}
 In this section we finish the proof of Theorem \ref{WarpConv2} which follows by the results of the last two subsections combined with Theorem \ref{HLS-thm}.
 
 \begin{proof}
 Let $p,q \in [-\pi,\pi]^3$ then by Lemma \ref{distLowerBound} we have
\begin{align}\label{distliminf2}
\liminf_{j \rightarrow \infty} d_{j}(p,q) \ge d_{\infty}(p,q).
\end{align}
By Lemmas \ref{ConstzControl}, \ref{ConstxyControl} or \ref{MonotoneControl} we have 
\begin{align}
\limsup_{j \rightarrow \infty} d_j(p,q)\le  d_{\infty}(p,q).
\end{align}

So by combining with \eqref{distliminf2} we conclude
\begin{align}
\lim_{j \rightarrow \infty} d_j(p,q) = d_{\infty}(p,q)\label{pointwiseDistConv2}
\end{align}
which gives pointwise convergence of distances.

Now by the assumption that $0 < c - \frac{1}{j} \le f_j \le K$ we can apply Lemma \ref{biLip} and choose $\lambda =\max \left(\frac{1}{\min(c/2,1)},\max(1,K) \right) > 0$ so that for $j$ chosen large enough we find
\begin{align}
\lambda \ge \frac{d_j(p,q)}{d_1(p,q)}\ge \frac{1}{\lambda}.\label{distLipBound4}
\end{align}
where $d_1$ is the distance defined with warping factor $1$.

Hence we can apply Theorem \ref{HLS-thm} to conclude that there exists a length metric $d_\infty'$ and a subsequence $d_{j_k}$ so that $d_{j_k}$ converges uniformly to $d_{\infty}'$, and GH and SWIF converges as well. By the pointwise convergence \eqref{pointwiseDistConv2} we know that $d_{\infty} = d_{\infty}'$ and hence $d_{j_k}$ must uniformly converge to $d_{\infty}$. Since this is true for all the subsequences, we see that $d_j$ uniformly converges to $d_{\infty}$ and hence Gromov-Hausdorff and intrinsic flat converges as well. 
 \end{proof}

 \bibliographystyle{alpha}
 \bibliography{Allen}
 
\end{document}